\documentclass[a4paper,oneside,11pt]{article}
\usepackage{hyperref}
\usepackage{enumitem}
\usepackage{amsmath,amsfonts,amscd,amssymb,amsthm}
\usepackage[margin=2cm]{geometry}
\usepackage{longtable}
\usepackage[english]{babel}
\usepackage[utf8]{inputenc} 
\usepackage[active]{srcltx}
\usepackage[T1]{fontenc}
\usepackage{graphicx}
\usepackage{pstricks}
\usepackage{bbm}
\usepackage{MnSymbol}
\usepackage{stmaryrd}
\usepackage{nicefrac}
\usepackage{calrsfs}
\usepackage{mathtools}
\usepackage{mdwlist}   
\usepackage{epstopdf} 

\newtheorem{theorem}{Theorem}[section]

\newtheorem*{claim*}{Claim}

\newtheorem{corollary}[theorem]{Corollary}
\newtheorem{lemma}[theorem]{Lemma}
\newtheorem{proposition}[theorem]{Proposition}

\newtheorem{remark}[theorem]{Remark}

\newcommand{\be}[1]{\begin{equation}\label{#1}}
\newcommand{\ee}{\end{equation}}
\numberwithin{equation}{section}

\newcommand{\ba}[1]{\begin{align}\label{#1}}
\newcommand{\ea}{\end{align}}
\numberwithin{equation}{section}

\newcommand{\ben}{\begin{equation*}}
\newcommand{\een}{\end{equation*}}
\numberwithin{equation}{section}

\renewenvironment{proof}[1][\relax]
  {\paragraph{Proof\ifx#1\relax\else~of #1\fi}}%
  {~\hfill$\square$\par\bigskip}

\newcommand{\calD}{\mathcal{D}}

\newcommand{\calW}{\mathcal{W}}

\newcommand{\bbC}{\mathbb{C}}

\newcommand{\bbR}{\mathbb{R}}

\newcommand{\bbZ}{\mathbb{Z}}

\newcommand{\ep}{\varepsilon}
\newcommand{\eps}{\ep}
\newcommand{\om}{\omega}
\newcommand{\Om}{\Omega}
\newcommand{\si}{\sigma}

\newcommand{\De}{\Delta}

\newcommand{\ovr}{\overline}

\newcommand{\Tr}{{\rm Tr}}

\newcommand{\rk}[1]{\bgroup\color{red}%
  \par\medskip\hrule\smallskip%
  \noindent\textbf{#1}%
  \par\smallskip\hrule\medskip\egroup}

\usepackage{wasysym}

\newcommand{\ind}{\mathbf{1}}

\usepackage{psfrag}
\def\mik{1}
\newcommand\cpsfrag[2]{\ifnum\mik=1\psfrag{#1}{#2}\fi}

\title{The Bethe ansatz for the six-vertex and XXZ models: an exposition}
\author{
	Hugo Duminil-Copin \thanks{I.H.\'E.S.} \thanks{University of Geneva}\ , 
	Maxime Gagnebin\addtocounter{footnote}{-1}\footnotemark\ , 
	Matan Harel\addtocounter{footnote}{-1}\footnotemark\ ,\\ 
	Ioan Manolescu\thanks{University of Fribourg}\ , 
	Vincent Tassion\addtocounter{footnote}{-2}\footnotemark}
\date{\today}

\begin{document}

\maketitle

\begin{abstract}
In this paper, we review a few known facts on the coordinate Bethe ansatz.
We present a detailed construction of the Bethe ansatz vector $\psi$ and energy $\Lambda$, which satisfy $V \psi = \Lambda \psi$, where $V$ is the the transfer matrix of the six-vertex model on a finite square lattice with periodic boundary conditions for weights $a= b=1$ and $c > 0$. We also show that the same vector $\psi$ satisfies $H \psi = E \psi$, where $H$ is the Hamiltonian of the XXZ model (which is the model for which the Bethe ansatz was first developed), with a value $E$ computed explicitly. 

Variants of this approach have become central techniques for the study of exactly solvable statistical mechanics models in both the physics and mathematics communities. Our aim in this paper is to provide a pedagogically-minded exposition of this construction, aimed at a mathematical audience. It also provides the opportunity to introduce the notation and framework which will be used in a subsequent paper by the authors \cite{BetheAnsatz2} that amounts to proving that the random cluster model on $\mathbb{Z}^2$ with cluster weight $q >4$ exhibits a first-order phase transition. 
\end{abstract}

\section{Introduction}
The study of statistical mechanics has greatly benefited from the analysis of exactly solvable lattice models. Although we will not offer a proper definition of the notion of exact solvability, its essence lies in the existence of closed-form formulae for many of the important thermodynamics quantities associated with the model. Perhaps the earliest example in modern statistical mechanics came in 1931, with Bethe's \cite{Bethe31} approach to diagonalizing the Hamiltonian of the XXZ model, a particular case of the anisotropic one-dimensional Heisenberg chain. His technique, now known as the {\em coordinate Bethe ansatz}, shows that, given a solution to a (relatively) small number of simultaneous nonlinear equations, one can construct a candidate eigenvector and eigenvalue - i.e. a vector $\psi$ satisfying $H \psi = E \psi$.  

In 1967, Lieb \cite{Lie67b} noticed that the same construction can be used to find candidate eigenvectors for the transfer matrix of the six-vertex model. This model, initially proposed by Pauling in 1931 for the study of the thermodynamic properties of ice, is a major object of study on its own right: see \cite{Resh10} and Chapter 8 of \cite{Bax89} (and references therein) for a bibliography on the six-vertex model. 

The work of Baxter \cite{Bax72} showed that there is a rich algebraic structure to the six-vertex model (as well as the eight-vertex model, which generalizes it). His approach, based on commuting matrices and the so-called Yang-Baxter relations, led to a great generalization of Bethe's original technique. This approach, called the {\em algebraic Bethe ansatz} (to distinguish it from the coordinate Bethe ansatz), has been at the heart of the study of exactly solvable models in the next two decades (see \cite{Jimbo} for a short survey of this work, and \cite{Korepin} for a more complete description). 

In this paper, we will focus on the original formulation -- the coordinate Bethe ansatz -- as it is sufficient for our analysis of the six-vertex model. Besides being a model of independent interest, this model is also deeply connected to Fortuin-Kasteleyn percolation, which motivates our work here; we defer a discussion of this connection to \cite{BetheAnsatz2}. As such, we will present a detailed derivation, aimed at a mathematical audience, of the construction of a candidate eigenvector for the six-vertex transfer matrix, under toroidal boundary conditions. 

Our goal will be to provide a proof of the statements which will be needed in the subsequent papers of this series. This include the ``singular'' case of the construction, in which one of the solutions to Bethe's equations is zero (see below for formal definitions). The paper ends with a short proof of the fact that the construction for the XXZ model is, in fact, identical to the one used for the six-vertex model. While these results are not new, we hope that an elementary exposition will nonetheless be of some use to the community.

\section{Definitions and statements of main theorems}\label{sec:notation}
\subsection{The six-vertex model and its transfer matrix}
For the rest of the paper, fix two positive integers $M$ and $N$.
Write $\bbZ_N$ and $\bbZ_M$ for the cyclic groups of order $N$ and $M$, respectively (which are identified with $\{1,\dots,N\}$ and $\{1,\dots,M\}$, respectively).
Consider the torus $\mathbb{T}_{N,M}$, with vertex set $ \bbZ_N \times \bbZ_M$ and edges between vertices at $\|\cdot\|_1$-distance 1 of each others.

Let $\omega$ be an arrow configuration on the edges of $\mathbb{T}_{N,M}$ -- i.e. a map $\omega$ from edge-set of $\mathbb T_{N,M}$ to $\{-1,1\}$, where $+1$ is considered as a right or up arrow, and -1 as a left or down arrow. The six-vertex model is given by restricting $\omega$ to configurations that have an equal number of arrows entering and exiting each vertex, i.e.~formally satisfying the {\em ice rule}
\[
\forall v \in \mathbb{T}_{N,M}, \quad \displaystyle \sum_{\begin{subarray}\ \text{edge }e\text{ with}\\\text{endpoint }v\end{subarray}} \omega(e) = 0 \, .
\]
The rule leaves six possible configurations at each vertex, depicted in Figure~\ref{fig:6vertices}. 
Assign the weight $a$ to configurations 1 and 2, $b$ to 3 and 4, and $c$ to 5 and 6. 
This choice is made to ensure that the weight is invariant under a global arrow flip.
Letting $n_i$ be the number of vertices with configuration $i$ in $\omega$, define the weight of $\om$ as
\[
w(\omega) = a^{n_1 + n_2} \cdot b^{n_3 + n_4} \cdot c^{n_5 + n_6} \, .
\]
Furthermore, if $\om$ does not obey the ice rule, set $w(\om) = 0$.  
The partition function of the model is given by
\[
Z_{6V}(a,b,c) = \displaystyle \sum w(\omega) \, , 
\]
where the sum is over all $4^{NM}$ arrow configurations, or equivalently, over all arrow configurations satisfying the ice rule. 

\begin{figure}[htb]
	\begin{center}
		\includegraphics[width=0.6\textwidth, page=1]{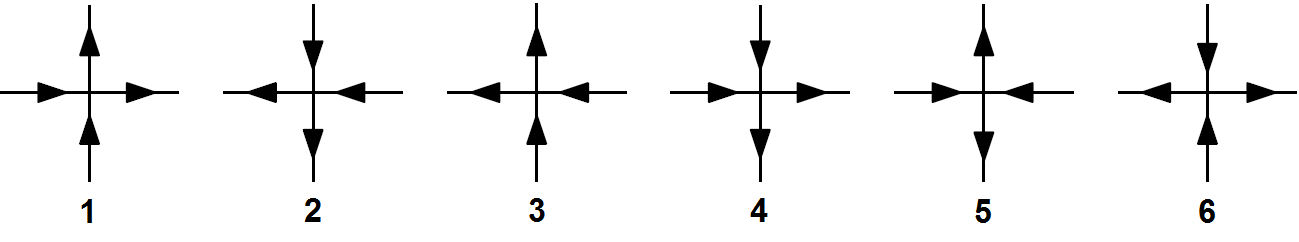}
	\end{center}
	\label{fig:6vertices}
	\caption{The six possibilities for vertices in the six-vertex models.
	Each possibility comes with a weight $a$, $b$ or $c$.}
\end{figure}

In this paper, we will study the isotropic model, in which $a = b =1$, while $c > 0$.
We note that similar statements can be formulated to generalize our work to arbitrary positive values of $a,b$ and $c$. 


\bigbreak
We now introduce a matrix $V$ which turns out to be the {\em transfer matrix} of the model (see the next section for more details).  Let $\vec x=(x_1,\dots,x_n)$ denote a set of orderd integers $1\le x_1<\dots<x_n\le N$ with $0\le n\le N$. The quantity $x_i$ will refer to the $i$-th coordinate of $\vec x$. We set $|\vec x|$ for the length $n$ of $\vec x$.

Let $\Omega = \{-1, 1\}^{\otimes N}$ be the $2^N$-dimensional real vector space spanned by the basis $\big\{\Psi_{\vec x}\big\}_{\vec x}$, where for any $\vec x$, $\Psi_{\vec x} \in \{\pm 1\}^{N}$ is given by
$$ 
	\Psi_{\vec x}(i) = 
	\begin{cases}
		+1 \quad \text{ if $i \in \{x_1,\dots, x_n\}$},\\
		-1 \quad \text{ if $i \notin \{x_1,\dots, x_n\}$}.
	\end{cases}
$$
Associate to each $\Psi_{\vec x}$ a sequence of vertical arrows entering or exiting the vertices of $\mathbb{Z}_N$, with up arrows at $x_i$ and down arrows otherwise. Note that $|\vec x|$ corresponds to the number of up arrows of $\Psi_{\vec x}$.

For two basis vectors $\Psi_{\vec{x}}, \Psi_{\vec y} \in \Omega$,  we say that $\Psi_{\vec x}$ and $\Psi_{\vec y}$ are {\em interlaced} if $|\vec{x}| = |\vec y|$ and  
\[
x_1 \leq y_1 \leq x_2\le \dots \leq x_n \leq y_n \quad\text{ or }\quad y_1 \leq x_1 \le y_2\leq \dots \leq y_n \leq x_n \, .
\]
For a pair of interlacing vectors, we define 
\[
	P(\Psi_{\vec{x}}, \Psi_{\vec{y}}) 
	 =| \{ i \in \bbZ_N :\, \Psi_{\vec{x}}(i) \neq \Psi_{\vec{y}}(i)\}|\, .
\]
The matrix $V$ is an endomorphism of $\Om$ written in the basis $(\Psi_{\vec{x}})_{\vec{x}}$. 
It is defined as follows:
\begin{align}\label{eq:V}
	V (\Psi_{\vec{x}}, \Psi_{\vec{y}}) = 
	\left\{\begin{array}{lr} 
		2 & \text{ if }\Psi_{\vec{x}} = \Psi_{\vec{y}}; \\ 
		c^{P(\Psi_{\vec{x}}, \Psi_{\vec{y}})}
		& \text{ if } \Psi_{\vec{x}} \neq \Psi_{\vec{y}} \text{ and } \Psi_{\vec{x}} \text{ and }  \Psi_{\vec{y}} \, \text{ are interlaced};\\ 
		0 & \text{otherwise}. 
	\end{array}\right. 
\end{align}

The spectral properties of this matrix encode many properties of the associated six-vertex model. As a motivating example, we will prove in Section~\ref{sec:V} one of the simplest such associations, between the trace of $V^M$ and the partition function of the six-vertex model.

\begin{proposition}\label{prop:transfer-matrix}
	$V$ is a block diagonal, symmetric matrix, fixing the subspaces
	\[
	\Omega_n := {\rm Span} \{ \Psi_{\vec{x}}: |\vec x|=n\} \, .  
	\]
	Furthermore, $Z_{6V}(1,1,c) = \Tr(V^M)$.
\end{proposition}

\subsection{Statement of the Bethe ansatz}\label{sec:stat-bethe-ansatz}
In light of the above proposition, we have a clear interest in studying the spectral properties of $V$ as these provide asymptotics for the partition function of the model (and other related quantities). 
For more precise statements, see \cite{BetheAnsatz2}.
The main theorem of this paper is the explicit construction of $\psi$ and $\Lambda$ such that $V \psi = \Lambda \psi$. This is the eponymous Bethe Ansatz, and its proof takes up the majority of this text. 

Set $\Delta := (2 - c^2)/2$, and define the function $S : \bbR^2 \to \bbC$ by
\[
	S(x,y) := e^{-ix} + e^{iy} - 2 \Delta \, . 
\] 
If $\De \in [-1,1)$, define $\mu$ to be the unique solution to $\cos(\mu) = -\Delta$, $\mu \in [0,\pi)$. For $\De < -1$, set $\mu =0$. We introduce the set $\mathcal{D}_\Delta := (-\pi + \mu, \pi - \mu)$. Next, we define $\Theta: \mathcal{D}_\Delta^2 \rightarrow \mathbb{R}$ to be the unique continuous function which satisfies $\Theta(0,0) = 0$ and
\[
	\exp(-i \Theta(x,y)) = e^{i(x-y)} \cdot \frac{S(x,y)}{S(y,x)} \, .
\]
It may be shown that such a function $\Theta$ exists and that it real and analytic on $\calD_\De$, for any $\De < 1$.
In this paper we will only use its differentiability, antisymmetry, 
and the algebraic relation~\eqref{eq:fundamental} which follows directly from the definition.

For $z\neq 1$, we set 
\begin{equation}\label{eq:LMRhoDef}
	L(z):= 1 + \frac{c^2 z}{1-z} \, , \qquad 
	M(z):= 1 - \frac{c^2}{1-z} \, . 
\end{equation}
For $|\vec x|=n$ and $(p_1, p_2, \dots, p_n) \in \mathcal{D}_\Delta^n$,  set 
	\begin{align}\label{eq:BA_eigenvector}
		\psi(\vec x) 
		:= \sum_{\sigma \in \mathfrak{S}_n} A_\sigma \prod_{k=1}^n \exp\left(i p_{\sigma(k)}x_k \right) \, ,
	\end{align}
	where $\mathfrak{S}_n$ is the symmetric group on $n$ elements and 
	\begin{align}\label{eq:A_si}
		A_\sigma := \varepsilon(\sigma)\,  \prod_{1 \leq k < \ell \leq n} e^{ip_{\sigma(k)}}\  S(p_{\sigma(k)},p_{\sigma(\ell)}),
		\qquad \text{for $\sigma \in \mathfrak{S}_n$},	
	\end{align}
	with $\varepsilon(\sigma)$ being the signature of the permutation. We also define the vector $\psi \in \Omega$ by
	\[ 
		\psi = \sum_{|\vec x|=n}\psi(\vec x)\, \Psi_{\vec x} \, .
	\]

\begin{theorem}[Bethe Ansatz for $V$]\label{thm:BA} 
	Fix $n \leq N/2$. 
	Let $(p_1, p_2, \dots, p_n) \in \mathcal{D}_\Delta^n$ be distinct and satisfy the equations
	\begin{align}\tag{BE}\label{eq:BA}
		\exp\left(i N p_j\right) 
		= (-1)^{n-1} \exp \left( -i\sum_{k=1}^n \Theta(p_j,p_k) \right),  \quad \forall j \in \{1, 2, \dots, n\}. 
	\end{align}
	Then, $\psi$ satisfies the equation $V\psi = \Lambda \psi$, where 
	\begin{align}\label{eq:BA_eigenvalue}
		\Lambda := 		
		\begin{cases}
			\displaystyle \prod_{j=1}^n L(e^{ip_j}) + \prod_{j=1}^n M(e^{ip_j}),\quad &\text{if $p_1,\dots,p_n$ are non zero,} \vspace{3pt}\\
			\displaystyle
			\Big[2+ c^2 (N-1) + c^2 \sum_{j\neq \ell} \partial_1 \Theta (0,p_{j}) \Big] \cdot \prod_{j \neq \ell} M(e^{ip_j}),
			\quad &\text{if $p_\ell = 0$ for some $\ell$.}
		\end{cases}
	\end{align}
\end{theorem}
We note that the restriction $n \leq N/2$ is insignificant since the transfer matrix $V$ is invariant under global arrow flip, and as a consequence the spectrums of $V$ on $\Omega_n$ and $\Omega_{N - n}$ are identical. 

\subsection{Comments on Theorem~\ref{thm:BA}}\label{sec:rr}
There are several important features of the theorem above which merit explicit mention:

\paragraph{Logarithmic form of the Bethe ansatz} 
This theorem reduces the ${N \choose n}$-dimensional problem of finding an eigenvector of $V$ in $\Omega_n$ to the solution of the $n$ relations~\eqref{eq:BA}, often called {\em Bethe's equations}. In most applications, it is far more instructive to consider the equations in their logarithmic form, i.e.  
\begin{align}\label{eq:BA-Logarthmic}
Np_j = 2 \pi I_j - \displaystyle \sum_{k=1}^n \Theta(p_j, p_k) \quad \forall j \in \{1, 2, \dots, n\} ,
\end{align} 
where $\{I_j\}$ are distinct integers (resp. half integers) if $n$ is odd (resp. even). 

\paragraph{Existence and uniqueness}
The existence of solutions to~\eqref{eq:BA} is nontrivial, and uniqueness is, in general, false (due to our ability to choose $\{I_j\}$ in the logarithmic form). It is more instructive to consider existence and uniqueness for~\eqref{eq:BA-Logarthmic}. In our subsequent paper, we consider a  specific choice of $\{I_j\}$ to prove existence of solutions which will generate the leading eigenvalues of $V$ restricted to $\Omega_{N/2 - k}$ for any fixed $k$. 

\paragraph{The coefficients $A_\sigma$ and the origin of the Bethe's equations \eqref{eq:BA}}
The function $A_\sigma$ is defined in such a way that the following relation holds true for every $1 \leq j < n$:
\begin{equation}\label{eq:fundamental}
\frac{A_{\sigma\circ (j,j+1)}}{A_\sigma}=-\exp\big(i\Theta(p_{\sigma(j)},p_{\sigma(j+1)})\big),
\end{equation}
where $(j,j+1)$ is the transposition permuting $j$ and $j+1$. The relations~\eqref{eq:BA} are introduced in order to obtain a similar identity for the transposition $(1,n)$ permuting $1$ and $n$
\begin{equation}\label{eq:fundamentalBoundary}
\frac{A_{\sigma\circ (1,n)}}{A_\sigma}=-\frac{\exp(iNp_{\si(n)})}{\exp(iNp_{\si(1)})} \times \exp\big( i\Theta(p_{\sigma(n)},p_{\sigma(1)})\big).
\end{equation}
This equation can be seen as enforcing toroidal boundary conditions. Those two relations are proved in Section~\ref{sec:encodWords}, and play a fundamental role in the proof.

\paragraph{The role of the singularity}
Inspecting the form of $L(z)$ and $M(z)$ clearly indicates that the case $p_\ell = 0$ for some $\ell$ requires special treatment. 
Moreover, solutions of~\eqref{eq:BA} in which $p_\ell = 0$ for some $\ell$ are not esoteric. In fact, the leading eigenvalue of $V$ restricted to $\Omega_n$ is given by such a solution whenever $n$ is odd. 

Note that the $p_\ell =0$ formula in~\eqref{eq:BA_eigenvalue} is not given by a simple limit of the formula in the line that precedes it; instead, it includes terms depending on the derivative of $\Theta$ that would have canceled out algebraically in the non-degenerate case. 

This degenerate case only appears when $a = b$. Theorem~\ref{thm:BA}  may be extended to a general six-vertex model by setting $\Delta = (a^2 + b^2 - c^2)/2ab$, and replacing $L$ and $M$ by $[ab + (c^2 - b^2) z]/[a^2 - abz]$ and $[a^2 - c^2 - abz]/[ab - b^2 z]$, respectively (setting $a=b = 1$ gives the formulation above). 
Then, whenever $a \neq b$, $L(e^{ip})$ and $M(e^{ip})$ are bounded for all values of $p$, thus eliminating the need for the singular case of Theorem~\ref{thm:BA}. 

\paragraph{Nontriviality of $\psi$}
It is important to note that the theorem does not guarantee that $\psi$ is a true eigenvector of $V$, as it may be identically equal to 0. In order to apply~\eqref{eq:BA} to deduce information on the spectrum of $V$, one must have an independent argument that ensures that $\psi$ is nonzero. A quick computation shows that, for $(p_1,\dots,p_n) \in \calD_\De^n$ with at least two equal entries, the vector $\psi$ given by~\eqref{eq:BA_eigenvector} is identically $0$; this explains the condition that $p_1,\dots,p_n$ be distinct. Again, specifying a set of $\{I_j\}$ in the logarithmic form of Bethe's equations is an essential step in applications, and usually enables one to prove that $\psi\ne 0$ on a case-by-case basis.

\subsection{The XXZ model}

The final result of this paper relates to the XXZ model on $\mathbb{Z}_N$, which describes a one-dimensional, periodic system of spin $1/2$ particles. For an introduction to the Bethe ansatz that is focused on the XXZ model and aimed at physicists, we refer the interested reader to the work of Karbach, Hu and M\"uller \cite{Karbach2,Karbach3,Karbach1}. Our goal here is not to present a detailed analysis of this model; instead, we will present a short proof that the Bethe ansatz vector $\psi$ of the six-vertex model is also useful for this {\em a priori} very different model.

To do so, we must first define the Hamiltonian $H$ of the XXZ model. We conserve the notation $\Omega$ as the vector space spanned by vertical arrow configurations on $\mathbb{Z}_N$. For any $i \in \mathbb{Z}_N$, let $w_i:\Omega \rightarrow \Omega$ be the linear operator which exchanges the arrows at $i$ and $i+1$, whenever they are different, and is zero otherwise. For $\De <1$, let
\begin{align}\label{eq:H_i}
	H_i (\Psi_{\vec{x}}, \Psi_{\vec{y}}) = 
	\left\{\begin{array}{lr} 
		\Delta/2 & \text{ if }\Psi_{\vec{x}} = \Psi_{\vec{y}}\, \text {and } \Psi_{\vec{x}}(i) = \Psi_{\vec{x}}(i+1);\\
		-\Delta/2 & \text{ if }\Psi_{\vec{x}} = \Psi_{\vec{y}}\, \text {and } \Psi_{\vec{x}}(i) \neq \Psi_{\vec{x}}(i+1);\\ 
		1 &\text{ if } w_i(\Psi_{\vec{x}})=  \Psi_{\vec{y}}; \\ 
		0 & \text{otherwise}. 
	\end{array}\right. 
\end{align}
The Hamiltonian $H$ is defined by 
\[
H:= \displaystyle \sum_{i=1}^N H_i \, .
\]

\begin{theorem}\label{thm:XXZ}
	Fix $\De < 1$ and $n \leq N/2$. 
	Assume $(p_1, p_2, \dots, p_n) \in \mathcal{D}_\Delta^n$ are distinct and satisfy~\eqref{eq:BA}
	and let $\psi$ be defined as in Theorem~\ref{thm:BA}. Then
	\[
	H \psi = E \psi \, ,
	\]
	where
	\[
	E =  \frac{N \Delta}{2} - 2 \sum_{k=1}^n [\Delta - \cos(p_k)] \, .
	\]
\end{theorem}

This result is simpler to prove once we know that $\psi$ is an eigenvector of $V^{(n)}$, rather than directly:
we will simply show that $H$ and $V^{(n)}$ commute, and therefore share eigenvectors. 
The exact value of $E$ will appear through direct computation.

\paragraph*{Organisation of the paper}
Section~\ref{sec:BA} contains the proof of Theorem~\ref{thm:BA}, the main result of this paper. In Section~\ref{sec:V} we prove Proposition~\ref{prop:transfer-matrix} on the structure and role of the transfer matrix. 
Finally, in Section~\ref{sec:XXZ}, we prove Theorem~\ref{thm:XXZ}, the equivalent of Theorem~\ref{thm:BA} for the XXZ model.

\paragraph*{Acknowledgements}
We thank A. Borodin, I. Corwin and C. Hagendorf for useful comments.
This research was supported by the NCCR SwissMAP, the ERC AG COMPASP, the Swiss NSF and the IDEX chair funded by  Paris-Saclay.

\section{Proof of the Bethe ansatz (Theorem~\ref{thm:BA})}\label{sec:BA}

We begin by introducing some useful notation. 
For the entire section, fix $n \leq N/2$ and some $(p_1, p_2, \dots, p_n) \in \mathcal{D}_\Delta^n$ where the $p_i$ are distinct. 
Set
\[
z_j = e^{i p_j} \in \{z \in \mathbb{C} : |z| = 1\}, \qquad \forall j \in \{1,\dots, n\}.
\]
Given $\sigma\in \mathfrak S_n$ and a vector $\vec x$, set  
\begin{equation}
  \label{eq:1}
  Z_\si^{\vec x}:=\prod_{j=1}^{n}z_{\si(i)}^{x_j}.  
\end{equation}
Also fix for the whole section a vector $\vec{x}$ with $|\vec x|=n$. 
%
Recall the definition of $\psi$ from the statement of the theorem.
Using the notation introduced above and the definition of interlacement, the coordinate $V \psi(\vec x)$ of the vector $V\psi$ along $\Psi_{\vec x}$  can be written as  
\begin{align}
	V \psi(\vec x)
	 =
	&\sum_{\sigma\in\mathfrak S_n}A_\sigma\sum_{1\le y_1\le x_1\le\cdots\le y_n\le x_n} 
		c^{ P(\Psi_{\vec{x}},\Psi_{\vec{y}}) } Z_{\si}^{\vec y}~+\sum_{\sigma\in\mathfrak S_n}A_\sigma\sum_{x_1\le y_1 \le \cdots\le x_n\le y_n\le N} c^{ P(\Psi_{\vec{x}},\Psi_{\vec{y}}) } Z_{\si}^{\vec y}.
	\label{eq:V_on_phi}
\end{align}
Note that the weight of $\vec{y}=\vec{x}$ is split up over both sums. Also keep in mind that the sums are on $\vec{y}=(y_1,\dots,y_n)$, where the $y_i$ are distinct and ordered.

One needs to show that the expression above is equal to $\Lambda \psi(\vec x)$. 
Our proof is organized as follows.
\begin{itemize}
\item In Section~\ref{sec:relat-satisf-a_sigma}, we state a lemma that provides several important relations satisfied by the coefficients $A_\si$, which will be used later in the proof.
\item In Section~\ref{sec:relations}, we provide a change of variables formula based on Bethe's equations. This formula allows us to ``merge'' the two terms in~\eqref{eq:V_on_phi} and provides us with the compact expression
\begin{equation}
  V\psi(\vec x)=\sum_{\si\in\mathfrak S_n} A_\si R_\si,\label{eq:4}
\end{equation}
where $R_\si$ will be defined as a sum suitable for further manipulations.

\item In Section~\ref{sec:encodWords}, we perform certain algebraic manipulations in order to rewrite the sum on the RHS of~\eqref{eq:4}. More precisely, we define a set of words $\mathcal W$ as well as functions $r_\si, Z_\si : \calW \to \bbC$ (defined in terms of the functions $L$, $M$ and the $z_i$'s) and prove an expression of the form
\begin{equation}\label{eq:2}
R_\si =\sum_{w \in \mathcal W}r_\si(w) Z_\sigma(w).
\end{equation}
The only requirement for this re-encoding step will be that none of the $z_i$'s is equal to 1 (or, equivalently, that none of the $p_i$'s is equal to $0$). This requirement is important when computing partial sums of the form  $\sum_{x_i\le y \le x_i}z_{\si(i)}^y$. 

\item In Section~\ref{sec:proofNonDegenerate}, we prove that $V \psi(\vec x)=\Lambda\psi(\vec x)$ in the non-singular case, when none of the $p_i$'s is equal to $0$. 
In this case the two steps~\eqref{eq:4} and~\eqref{eq:2} are  valid and we can write
\begin{equation*}
  V \psi(\vec x)=\sum_{w \in \mathcal W} \sum_{\si\in\mathfrak S_n} A_\sigma r_\si(w) Z_\sigma(w).
\end{equation*}
We conclude the proof by showing that for any non-constant $w \in \mathcal{W}$, 
\begin{equation}\label{eq:CancellationNonsingular}
\sum_{\si\in\mathfrak S_n} A_\sigma r_\si(w) Z_\sigma(w) = 0.
\end{equation}
The remaining terms corresponding to constant words will be equal to $\Lambda\psi(\vec x)$.       
\item In Section~\ref{sec:proofDegenerate}, we treat the singular case when one entry $p_\ell=0$. In this case, the encoding with words~\eqref{eq:2} is not valid directly (since $z_\ell=1$). Nevertheless, we will be able to perform a perturbative strategy, and write 
\begin{equation*}
  V \psi(\vec x)=\lim_{\eps\to 0}\sum_{w \in \mathcal W} \sum_{\si\in\mathfrak S_n} A_\sigma r_\si^\eps(w) Z^\eps_\sigma(w)
\end{equation*} 
where $r_\si^\eps$ and $Z_\si^\eps$ are defined by replacing $p_\ell(=0)$ with $\eps$ in the definition of $r_\si$ and $Z_\si$.
When analysing the contribution (as $\eps\rightarrow0$) of words in $\mathcal W$, we will need to keep track of the first order terms (terms of order $\eps$) which compensate diverging terms of the form $L(e^{i\eps})$ or $M(e^{i\eps})$ and do not vanish in the limit: this explains the different expression of $\Lambda$ when one of the $p_i$'s is 0.
\end{itemize}

\subsection{Relations satisfied by the coefficients $A_\sigma$}
\label{sec:relat-satisf-a_sigma}

The coefficients $A_\si$ defined in~\eqref{eq:A_si} play an important role in our proof. In order to perform algebraic manipulations, one needs to express $A_{\si \circ \si'}$ as a function of $A_\si$ for certain permutations $\si'$. Furthermore, the coefficients $A_\sigma$ are related to the functions $L$ and $M$ introduced in Section~\ref{sec:stat-bethe-ansatz}. In the lemma below, we state the relations needed for our derivation of the Bethe ansatz.

Let $\tau$ be the permutation with $\tau(i) = i+1$ for $1 \leq i < n$ and $\tau(n) = 1$. 
Moreover, let $(j,k)$ be the transposition inverting the elements $j$ and $k$. 

\begin{lemma}\label{lem:ratios} 
Assume that $(p_1,\ldots,p_n)\in\mathcal D_{\Delta}^n$ satisfies Bethe's equations \eqref{eq:BA}, and let $A_\si$ be the coefficients defined by~\eqref{eq:A_si}. Then, for every $\sigma \in \mathfrak S_n$, we have
\begin{align}
 \frac{A_{\sigma\circ (j,j+1)}}{A_\sigma}&=-\exp\big(i\Theta(p_{\sigma(j)},p_{\sigma(j+1)})\big) \quad\text{for every $1\le j<n$,}\label{eq:8}
\\
\frac{A_{\sigma\circ (n,1)}}{A_\sigma}&= -\left(\frac{z_{\sigma(n)}}{z_{\sigma(1)}}\right)^{N} \cdot\exp\big( i\Theta(p_{\sigma(n)},p_{\sigma(1)})\big),\label{eq:9}\\
\frac{A_{\sigma\circ\tau}}{A_\sigma}&= z_{\si(1)}^{-N}.\label{eq:11}
\end{align}
Furthermore, when  $p_j$ and $p_k$ are nonzero, 
\begin{equation}
  \label{eq:6}
  \exp\big(i\Theta(p_{j},p_{k})\big)=\frac{M(z_j)L(z_k)-1}{M(z_k)L(z_j)-1}.
\end{equation}
\end{lemma}
\begin{remark}
  Equations~\eqref{eq:8} and~\eqref{eq:6} do not use that the $p_k$'s are solutions of Bethe's equations. 
\end{remark}

\begin{proof}
Equations~\eqref{eq:8} and~\eqref{eq:6} are straightforward consequences of the definitions of $A_\si$ and $\Theta$. In order to prove~\eqref{eq:11}, we use the transposition decomposition $\tau=(1,2)\circ\dots\circ(n-1,n)$ and apply~\eqref{eq:8} $n-1$ times to deduce
\begin{align*}
	\frac{ A_{(\sigma \circ \tau)}}{A_\sigma}& = (-1)^{n-1}\cdot\exp\left(i \sum_{k=2}^n \Theta\left(p_{\sigma(1)},p_{\sigma(k)}\right)\right)= (-1)^{n-1}\cdot\exp\left(i \sum_{k=1}^n \Theta\left(p_{\sigma(1)},p_{k}\right)\right).
\end{align*}
In the last equality, we used the fact that $\Theta(p,p)=0$ to add the missing term in the sum. Therefore,
\begin{equation*}
\frac{A_{(\sigma \circ \tau)} z_{\sigma(1)}^N}{A_\sigma}  = (-1)^{n-1}\exp\left(i p_{\sigma(1)} N + i \sum_{k=1}^n \Theta\left(p_{\sigma(1)},p_{k}\right)\right)  
\stackrel{\eqref{eq:BA}}{=}1,
\end{equation*}
which proves~\eqref{eq:11}.

Finally, we deduce~\eqref{eq:9} from the previous computations by using the decomposition $(n,1) = \tau^{-1}\circ(1,2)\circ\tau$. We find
\begin{align*}
\frac{A_{\sigma\circ (n,1)}}{A_\sigma} & = \frac{A_{\sigma\circ \tau^{-1}}}{A_\sigma} \cdot \frac{A_{\sigma\circ \tau^{-1} \circ (1,2)}}{A_{\sigma\circ \tau^{-1}}} \cdot \frac{A_{\sigma\circ \tau^{-1} \circ (1,2) \circ \tau}}{A_{\sigma\circ \tau^{-1} \circ (1,2)}} \\ 
& = z_{\sigma(n)}^N \cdot\left[ -\exp\big(i\Theta(p_{\sigma\circ\tau^{-1}(1)},p_{\sigma\circ\tau^{-1}(2)})\big) \right] \cdot z_{(\sigma \circ \tau^{-1} \circ (1,2))(1)}^{-N} \\
 & =  -\left(\frac{z_{\sigma(n)}}{z_{\sigma(1)}}\right)^{N} \cdot \exp\big(- i\Theta(p_{\sigma(1)},p_{\sigma(n)})\big)\\
& =  -\left(\frac{z_{\sigma(n)}}{z_{\sigma(1)}}\right)^{N} \cdot \exp\big(  i\Theta(p_{\sigma(n)},p_{\sigma(1)})\big),
\end{align*}
where in the last line we used the antisymmetry of $\Theta$. 
\end{proof}

\subsection{Toroidal boundary conditions}\label{sec:relations}

As mentioned in the first comment of Section~\ref{sec:rr}, Bethe's equations~\eqref{eq:BA} implies the  important ``boundary relation''~\eqref{eq:fundamentalBoundary} between the coefficients $A_\sigma$'s. This relation  will allow us to perform a change of variables, stated in the proposition below, that will be instrumental in our proof. 

To express this formula in a compact way, set $x_0 = x_N - N \leq 0$. Recall that
\begin{equation}\label{eq:abc}
	c^{P(\Psi_{\vec x},\Psi_{\vec y})}=\prod_{k=1}^n c^{2\mathbf 1_{y_k\notin \{x_{k-1},x_k\}}}
\end{equation}
for any $\vec y$ interlaced with $\vec x$. 
We extend this formula to all sets $(y_1,\dots, y_n)$ with $x_0\le y_1\le x_1\le \dots\le x_{n-1}\le y_n\le x_n$. 
Henceforth $\vec y$ is considered to satisfy the more relax condition above; in particular, we may have $y_1 < 0$ 
and $y_j = y_{j+1}$ for certain $j$'s.

Recall that $\tau$ is the cyclic permutation of $\{1,\dots,n\}$ defined by $\tau(n)=1$ and $\tau(i)=i+1$ for each $1\leq  i< n$. 

\begin{proposition}[Change of variables formula]\label{prop:changeOfVariables} 
	Assume that  $(p_1,\ldots,p_n)\in\mathcal D_{\Delta}^n$ satisfies Bethe's equations~\eqref{eq:BA}.
	For any function $f:\mathfrak S_n\rightarrow \bbR$, 
	\begin{equation}\label{eq:5} 
		\sum_{\sigma\in\mathfrak S_n}A_\sigma\, f(\sigma)\, Z_\si^{\vec x}
		=\sum_{\sigma\in\mathfrak S_n} A_\sigma\, f(\sigma\circ\tau)\, Z_{\si}^{\tau^{-1}\vec x},
	\end{equation}
	where $\tau^{-1}\vec x:= (x_0,x_1,\ldots,x_{n-1})$. 
\end{proposition}

\begin{proof}
  By making the change of variables $\sigma\mapsto\sigma\circ \tau$, the LHS of~\eqref{eq:5} is equal to
\begin{equation}\label{eq:CyclicRewrite}
  \sum_{\sigma\in\mathfrak S_n} A_{(\sigma\circ \tau)}  f(\sigma\circ \tau) \:  Z_{(\si\circ \tau)}^{\vec x}.
\end{equation}
Then, ~\eqref{eq:11} and  the straightforward computation $Z_{\si\circ \tau}^{\vec x} = z_{\si(1)}^N Z_\si^{\tau^{-1}\vec x}$ complete the proof.
\end{proof}

\begin{corollary}\label{cor:expressionRsi} If $p_1,\ldots,p_n$ are solutions of Bethe's equations \eqref{eq:BA}, then
  \begin{equation}\label{eq:lem0}
V \psi(\vec x)
	 ~=~
	\sum_{\sigma\in\mathfrak S_n}A_\sigma\underbrace{\sum_{x_0\le y_1\le x_1\le y_2\le\dots\le x_{n-1}\le y_n\le x_n} 
		c^{ P(\Psi_{\vec{x}},\Psi_{\vec{y}}) }Z_\si^{\vec y}}_{\displaystyle R_\si},
\end{equation}
where the second sum is such that the $y_i$'s must all be distinct modulo $N$. 
\end{corollary}

\begin{proof}
  Consider the two terms on the RHS of~\eqref{eq:V_on_phi}. By applying the change of variables formula of Proposition~\ref{prop:changeOfVariables} and reindexing $\vec y$, the second term is equal to
$$ \sum_{\sigma\in\mathfrak S_n}A_{\sigma} \sum_{x_0\le y_1\le 0 \le x_1\le y_2\le \dots\le x_{n-1}\le y_{n}\le x_n} 
c^{ P(\Psi_{\vec{x}},\Psi_{\vec{y}}) }Z_{\si(k)}^{\vec y}.$$ Then, combining this expression with the first term on the RHS of~\eqref{eq:V_on_phi} yields the desired expression for $(V\psi)(\vec x)$.
\end{proof}

\subsection{Encoding with words}

\label{sec:encodWords}

The goal of this section is to provide an alternate sum representation for $R_{\sigma}$. 
For this part, we do not assume that $(p_1,\dots, p_n)$ satisfy~\eqref{eq:BA}.

Computing directly $R_{\sigma}$ is rather cumbersome, due to the restriction forcing the $y_i$'s to be distinct. To illustrate the fact that the restriction on the $y_i$'s to be distinct creates the main difficulty, let us start by computing a slightly different quantity obtained by considering the sum $R_\sigma(\emptyset)$ (the notation will become clear later) of the expression $c^{ P(\Psi_{\vec{x}},\Psi_{\vec{y}}) }\prod_{k=1}^{n} z_{\si(k)}^{y_k}$ {\em for all} $y_1,\dots, y_n$ with $x_0\le y_1 \leq x_1 \leq y_2 \dots \leq x_n$ -- even when the terms $y_1,\dots,y_n$ are not distinct modulo $N$
(the notation $c^{ P(\Psi_{\vec{x}},\Psi_{\vec{y}})}$ in this case was defined in~\eqref{eq:abc}).
Recall that $\vec x$ is fixed and that the sum below is only on the $y_i$.
\begin{align}
	R_\sigma(\emptyset)&:=\sum_{x_0\leq y_1\leq x_1\leq y_2 \leq \dots \leq x_n} c^{ P(\Psi_{\vec{x}},\Psi_{\vec{y}}) }\   Z_{\si}^{\vec y} \nonumber\\
	&=\sum_{x_0\leq y_1\leq x_1\leq y_2 \leq \dots \leq x_n}  \prod_{k=1}^{n} c^{2\mathbf 1_{y_k\notin \{x_{k-1},x_k\}}}\,z_{\si(k)}^{y_k} \nonumber\\
	&= \sum_{x_0\leq y_1\leq x_1}c^{2\mathbf 1_{y_1\notin \{x_0,x_1\}}}\,z_{\si(1)}^{y_1}\sum_{x_1\leq y_2\leq x_2}c^{2\mathbf 1_{y_2\notin \{x_{1},x_2\}}}\,z_{\si(2)}^{y_2}\dots\sum_{x_{n-1}\leq y_n\leq x_n}	c^{2\mathbf 1_{y_n\notin \{x_{n-1},x_n\}}}\,z_{\si(n)}^{y_n} \nonumber \\
		&  = \prod_{k=1}^n \Big[ z_{\si(k)}^{x_{k-1}} +~  \sum_{x_{k-1} <y <x_k }  c^2z_{\si(k)}^{y} ~+ z_{\si(k)}^{x_{k}} \Big] \nonumber \\
			&=\prod_{k=1}^n \big[L(z_{\si(k)})z_{\si(k)}^{x_{k-1}} ~+~ M(z_{\si(k)}) z_{\si(k)}^{x_k}\big] \label{eq:EmptySetCalc}.
			\end{align}
The last equality is given by the definition of $L,M$, and the basic formula on geometric series. Note that this equality only holds if the values of $z_1,\dots, z_n$ are all different from $1$ - i.e. no $p_1,\dots, p_n$ is equal to $0$. 

Let $\mathcal W:=\{L,M\}^n$. An element of $\mathcal W$ is called a \emph{word} and is denoted by $w=w_1w_2\ldots w_n$. Expanding the product in~\eqref{eq:EmptySetCalc}, we get
\begin{align}\label{eq:aa}
R_\sigma(\emptyset)&=\sum_{w\in\mathcal W}\prod_{k:\, w_k = L} L(z_{\si(k)})z_{\si(k)}^{x_{k-1}} \cdot \prod_{k:\,w_k = M}M(z_{\si(k)})z_{\si(k)}^{x_k}.
\end{align}
\medbreak Importantly, the separation of the sums on the $y_i$ was possible because we dropped the restriction that $y_1,\dots, y_n$ be distinct. In the general case, this is not possible; however, we will still manage to express $R_\si$ using the strategy above via an inclusion-exclusion formula.

For $w \in \calW$, define
\begin{align}
  &r_\si(w) := \prod_{k\::\:w_kw_{k+1}=ML} (M(z_{\si(k)})L(z_{\si(k+1)})-1) \ \prod_{k\::\:w_{k-1}w_{k}=LL}  L(z_{\si(k)})  \prod_{k\::\:w_kw_{k+1}=MM}M(z_{\si(k)}), \label{eq:rSigW}  \\
&Z_\si(w):=\prod_{k:\, w_k = L} z_{\si(k)}^{x_{k-1}} \cdot \prod_{k:\,w_k = M}z_{\si(k)}^{x_k} \nonumber.
\end{align}
(Note that the indexes are $k-1$ and $k$ in the second product of the definition of $r_\sigma(w)$, and $k$ and $k+1$ in the third.) The next lemma shows that $R_\sigma$ itself can be written in terms of the quantities $r_\si(w)$ and $Z_\si(w)$.
\begin{lemma}\label{lem:1}
For any $\sigma\in \mathfrak{S}_n$ and any $p_1,\dots,p_n$ distinct and non-zero, 
$$R_\sigma=\sum_{w\in \mathcal W}r_\si(w) Z_\si(w).$$
\end{lemma}

\begin{proof}[Lemma~\ref{lem:1}]
For $S\subset\{1,\dots,n\}$, introduce
$$R_\sigma(S):=\sum_{\substack{ x_0\leq y_1\leq x_1\leq y_2 \leq \dots  \leq x_n \\i \in S \Rightarrow  y_{i} = x_i = y_{i+1} }} c^{ P(\Psi_{\vec{x}},\Psi_{\vec{y}}) }\ Z^{\vec y}_\si,$$
where $y_{n+1} = y_1 + N$. The definition is coherent with the quantity $R_\sigma(\emptyset)$ introduced before the lemma. Note that $R_\sigma(S)=0$ as soon as $S$ contains two successive integers (with $n+1$ being identified with 1) since $x_i\ne x_{i+1}$.
Thus, we will assume henceforth that $S$ contains no two two successive integers. 
With this notation, the inclusion exclusion formula reads:
\begin{align}\label{eq:RS}
	R_{\sigma} = \displaystyle \sum_{S \subset\{1,\dots,n\}} (-1)^{|S|} R_{\sigma}(S) \, .
\end{align}
Now, the computation leading to~\eqref{eq:EmptySetCalc} can be repeated for $S\neq \emptyset$ to yield
\begin{align}\label{eq:ArbitrarySetCalc}
R_\sigma(S) =\prod_{k:\{k-1,k\}\cap S=\emptyset}^n \big[L(z_{\si(k)})z_{\si(k)}^{x_{k-1}} ~+~ M(z_{\si(k)}) z_{\si(k)}^{x_k}\big]\  \prod_{k \in S} z_{\si(k)}^{x_k}\ \prod_{k\::\: k-1 \in S}z_{\si(k)}^{x_{k-1}}.
\end{align}
This is because, whenever $k \in S$, the sums over $y_k$ and $y_{k+1}$ are degenerate, including only one term - namely $z_{\si(k)}^{x_k}$. (The condition $k-1\in S$ in the last product corresponds to the fact that $y_{k+1}$ is equal to $x_k$ when $k\in S$.) Meanwhile, unrestricted $y_k$'s result in geometric sums, as before. 

Fix $S \subset \{1,\dots,n\}$ with no two consecutive values when considered periodically.
As in~\eqref{eq:aa}, we may expand the first product in terms of words $w = w_1\dots w_n$. 
However, only the choices of letters $w_i$ with $i \notin S$ and $i-1 \notin S$ matter. 
Thus, we expand $R_\si(S)$ using words $w \in \{L,M\}^n$, with the restriction that $w_k=M$ and $w_{k+1}=L$ for $k\in S$;
the choice of $w_k$ when $\{k-1,k\}\cap S=\emptyset$ is free and indicates whether we pick the term $L(z_{\sigma(k)})z_{\sigma(k)}^{x_{k-1}}$ or $M(z_{\sigma(k)})z_{\sigma(k)}^{x_k}$ in the first product. 
For a word $w$, write $S(w)$ for the set of indices $k$ such that $w_kw_{k+1}=ML$. 
Then the above restriction may be written as $S(w)\supset S$. Therefore, 
\begin{align*}
   R_\sigma(S) &=  \sum_{\substack{w\in \mathcal W : \\ S(w)\supset S}}\ \prod_{\substack{k:\, w_k = L\\\{k-1,k\}\cap S=\emptyset}} L(z_{\si(k)})z_{\si(k)}^{x_{k-1}}  \prod_{\substack{k:\,w_k = M\\\{k-1,k\}\cap S=\emptyset}}M(z_{\si(k)})z_{\si(k)}^{x_k}\  \prod_{k \in S} z_{\si(k)}^{x_k} \ \prod_{k\::\: k-1 \in S}z_{\si(k)}^{x_{k-1}}\\
   &= \sum_{\substack{w\in \mathcal W:\\ S(w)\supset S}}\ \bigg[
   \prod_{\substack{k:\,w_k = L\\k-1\notin S}} L(z_{\si(k)})  
   \prod_{\substack{k:\,w_k = M\\k\notin S}}M(z_{\si(k)})\bigg] Z_\si(w).
\end{align*}
In the second line, we have used that the considered words satisfy $w_k=M$ and $w_{k+1}=L$ for all $k\in S$.
Plugging this expression in~\eqref{eq:RS} and interchanging the sums, we find
\begin{equation}
  \label{eq:15}
  R_\si=\sum_{w\in\mathcal W}\bigg[\sum_{S\subset S(w)}(-1)^{|S|} \prod_{\substack{k:\, w_k = L\\k-1\notin S}} L(z_{\si(k)})  \prod_{\substack{k:\,w_k = M\\k\notin S}}M(z_{\si(k)})\bigg] Z_\si(w),
\end{equation}
where $S$ ranges over sets with no two consecutive values. 
In order to conclude the proof, one needs to check that the term inside the brackets in the equation
above is equal to $r_\si(w)$. To see this, expand the first product in the definition of $r_\si(w)$ (see~\eqref{eq:rSigW}) in order
to obtain
\begin{align*}
   r_\si(w)&=  \sum_{S \subset S(w)} (-1)^{|S|} \prod_{k \in S(w) \setminus S} M(z_{\si(k)}) L(z_{\si(k+1)}) \prod_{k\::\:w_{k-1}w_{k}=LL}  L(z_{\si(k)})  \prod_{k\::\:w_kw_{k+1}=MM}M(z_{\si(k)}). 
\end{align*}
One may check that the expression above matches the bracketed term in~\eqref{eq:15}, and this completes the proof.
\end{proof}

\subsection{Proof of Theorem~\ref{thm:BA} when no entry is zero}\label{sec:proofNonDegenerate}
In this section assume $(p_1,\ldots,p_n)\in\mathcal D_{\Delta}^n$ satisfies the Bethe equations~\eqref{eq:BA} and that $p_k\neq0$ for every $k$. We will prove that $V\psi(\vec x)=\Lambda\psi(\vec x)$. From Corollary~\ref{cor:expressionRsi} and Lemma~\ref{lem:1} (which can be applied since the $p_k$'s are nonzero), we already know that

\begin{equation}
	V \psi(\vec x) = \sum_{w\in \mathcal W}\sum_{\sigma\in \mathfrak S_n}  A_\sigma r_\si(w)Z_\si(w).\label{eq:14}
\end{equation}

We begin with an important lemma, proving that the sum above has many cancellations and reduces to a sum over exactly two words:
\begin{lemma}\label{lem:2}
  Let $\mathcal W_0=\{L\cdots L, M\cdots M\}$ be the set of constant words. Then,
\[
V \psi(\vec x) = \sum_{w \in \mathcal W_0} \sum_{\sigma\in \mathfrak S_n} A_\sigma r_\si(w)Z_\si(w).
\]
\end{lemma}
\begin{proof}
  Thanks to~\eqref{eq:14}, it is sufficient to show that for any $w \in \mathcal{W} \setminus \mathcal{W}_0$,
\[
\sum_{\sigma\in \mathfrak S_n} A_\sigma r_\si(w) Z_{\si}(w) = 0. 
\]
Fix a particular word $w \in \mathcal{W} \setminus \mathcal{W}_0$, and pick some $m$ such that $w_mw_{m+1}=ML$ (we consider the integers modulo $n$, in particular $n+1$ is identified with $1$). By pairing the permutation $\si$ with $\si\circ(m,m+1)$, we can write the sum displayed above as 
\begin{equation}
	\frac{1}{2} \sum_{\sigma\in \mathfrak S_n} 
	\left[A_\sigma r_\si(w) Z_{\si}(w) + A_{\sigma \circ (m, m+1)} r_{\sigma \circ (m, m+1)}(w) Z_{\si\circ (m, m+1)}(w)\right]. \label{eq:20}
\end{equation}
We wish to compute the ratio of the two terms in the summand above. First, it follows from the definitions of $r_\si(w)$ and $Z_\si(w)$  that 
\begin{equation*}
   \frac{ r_{\sigma \circ (m, m+1)}(w)}{r_\si(w)} =  \frac{M(z_{\si(m+1)})L(z_{\si(m)})-1}{M(z_{\si(m)})L(z_{\si(m+1)})-1}\quad\text{and}\quad  \frac{Z_{\si\circ (m, m+1)}(w)}{Z_{\si}(w)}=\left(\frac{z_{\sigma(1)}}{z_{\sigma(n)}}\right)^{N\mathbf 1_{m=n}}.
\end{equation*}
Furthermore, by Lemma~\ref{lem:ratios},  we have
\begin{equation*}
   \frac{A_{\sigma \circ (m, m+1)}}{A_\si}=-\left(\frac{z_{\sigma(n)}}{z_{\sigma(1)}}\right)^{N\mathbf 1_{m=n}} \frac{M(z_{\si(m)})L(z_{\si(m+1)})-1}{M(z_{\si(m+1)})L(z_{\si(m)})-1}.
\end{equation*}
Therefore, for any value of $m$ 

\begin{equation}
\frac{A_{\sigma \circ (m, m+1)} r_{\sigma \circ (m, m+1)}(w) Z_{\si\circ (m, m+1)}(w)}{A_\sigma r_\si(w) Z_{\si}(w)} = -1 ,\label{eq:18}
\end{equation}
so that the sum~\eqref{eq:20} vanishes.
\end{proof}

We conclude the proof by computing the contributions corresponding to the constant words in the simple expression of $V\psi(\vec x)$ provided by the previous lemma. The definition of $Z_{\sigma}^{\vec x}$ implies that
$Z_\si(M\cdots M) = Z_{\sigma}^{\vec x}$ and $Z_\si(L\cdots L) = Z_{\sigma}^{(\tau^{-1}) \cdot \vec x}.$

Hence, the sum corresponding to the word $M\cdots M$ is equal to
\[
\sum_{\sigma\in \mathfrak S_n}A_\sigma r_\sigma(M\cdots M) Z_\si(M\cdots M) = \Big(\prod_{i=1}^n M(z_i) \Big)\sum_{\sigma\in \mathfrak S_n}A_\sigma Z_{\sigma}^{\vec x} = \Big(\prod_{i=1}^n M(z_i) \Big)\,\psi(\vec x).
\]
For the word $L \cdots L$, the same computation gives
\[
\sum_{\sigma\in \mathfrak S_n}A_\sigma r_\si(L\cdots L) Z_\si(L \cdots L) = \Big(\prod_{i=1}^n L(z_i) \Big)\sum_{\sigma\in \mathfrak S_n}A_\sigma Z_{\sigma}^{\tau^{-1}\vec x} = \Big(\prod_{i=1}^n L(z_i)\Big)\,
\psi(\vec x),
\]
where the final equality follows from the change of variables formula~\eqref{eq:5}.

\subsection{Proof of Theorem~\ref{thm:BA} when one entry is zero}\label{sec:proofDegenerate}

For this part, suppose $(p_1, p_2, \dots, p_n) \in \mathcal{D}_\Delta^n$ satisfies the Bethe equations~\eqref{eq:BA} 
and that one of $p_1,\dots, p_n$ is null. 
Since $p_1,\dots,p_n$ are distinct, there exists exactly one index $\ell$ with $p_\ell = 0$.
The symmetry under the permutation group allows us to assume without loss of generality that $p_1= 0$. 
Henceforth we work under this assumption.

In the whole proof, we consider integers modulo $n$. In particular, $n+1$ is considered equal to 1. Recall that $\mathcal W_0$ denotes the set of constant words, and introduce the set $\mathcal W_1$ of words $w$ such that there exists a unique index $m$ with $w_mw_{m+1}=ML$. These words are formed, when regarded periodically, from a non-empty sequence of letters $M$ followed by a non-empty sequence of letters $L$. We also set
$$\Pi_M:= \prod_{k=2}^nM(z_k).$$

The proof begins very much like in the previous section. Namely, we may apply~\eqref{eq:lem0}, as it does not rely on the assumption that the $p_i$'s are nonzero to find
$$V\psi(\vec x)=\sum_{\sigma\in\mathfrak S_n}A_\sigma R_\sigma.$$
The computation of $R_\si$ in Section~\ref{sec:encodWords} was based on the assumption that the $p_k$'s are non-zero. To reuse those results, we introduce a new variable $\varepsilon \not \in \{0,p_2,\dots,p_n\}$ and set $z=\exp(i \varepsilon)$. Our goal will be to take the limit as $\varepsilon$ tends to 0 of the quantities defined below. 

Let $R_\sigma^{\varepsilon}$, $r^\varepsilon_\si(w)$ and $Z_{\sigma}^{\varepsilon}(w)$ be the quantities defined in the previous section, but with $\varepsilon$ instead of $p_1(=0)$ and therefore $z$ instead of $z_1=\exp(ip_1)=1$. Lemma~\ref{lem:1} (which does not rely on~\eqref{eq:BA}) gives
\begin{align}\label{eq:REpsilonRep}
\displaystyle R_\sigma^{\varepsilon} = \sum_{w\in\mathcal W}  r^\varepsilon_\si(w) Z_\sigma^\varepsilon(w).
\end{align}

Observe that $R^{\varepsilon}(\sigma)$ is a polynomial in $z$ and that it is equal to $R_{\sigma}$ when $z=1$. Thus, continuity guarantees that
\[
V \psi (\vec x) =  \lim_{\varepsilon \rightarrow 0} \sum_{\si\in \mathfrak S_n} A_\sigma R_\sigma^{\varepsilon}.
\]

Note that the coefficients $A_\sigma$ used here {\em do not} depend on $\varepsilon$; they are computed using $(p_1,\dots,p_n)$. Before studying the limit when $\eps$ tends to $0$, we use the word encoding of $R_\si^\eps$ and perform some algebraic manipulations as in the nonsingular case in order to obtain a simple expression for $\sum_\si A_\si R_\si^\eps$. 

Summing~\eqref{eq:REpsilonRep} over all the permutations, 
we obtain
\begin{equation}
  \label{eq:10}
  \sum_{\si\in \mathfrak S_n} A_\si R^\eps_\si=\sum_{\substack{(w, \si )\in \mathcal W\times \mathfrak S_n}} g^\eps_\si(w), 
  \qquad  \text{ where $g^\eps_\si(w) :=  A_\si r^\varepsilon_\si(w) Z_\sigma^\varepsilon(w)$} .
\end{equation}

We begin by applying the strategy from the proof of Lemma~\ref{lem:2}. Using suitable pairing, we obtain many cancellations in the sum above. This is based on the following relation. Let $w\in \mathcal W$ and $\si\in \mathfrak S_n$, and assume that there exists an index $m$ such that $w_mw_{m+1}=ML$ and $\sigma(m)$ and $\sigma(m+1)$ are different from 1. In this case, as in \eqref{eq:18}, we have
\begin{equation*}
\frac{g_{\si\circ(m,m+1)}^\eps(w)} {g_{\si}^\eps(w)}=-1.
\end{equation*}
Thus, the contribution of any pair $(w,\si)$ to~\eqref{eq:10} cancels out with that of $(w,\si\circ(m,m+1))$.
The only terms in~\eqref{eq:10} that do not vanish correspond to pairs $(w,\sigma)$ such that 
\begin{enumerate}[noitemsep,nolistsep]
\item[a.] $w\in \mathcal W_0$ and $\sigma\in\mathfrak S_n$ is arbitrary or 
\item[b.] $w\in\mathcal W_1$ and $\sigma\in\mathfrak S_n$ satisfies that $\sigma(m)=1$ or $\sigma(m+1)=1$ for the unique $m$ such that $w_mw_{m+1}=ML$.
\end{enumerate} 
We obtain



\begin{equation}
  \sum_{\si\in \mathfrak S_n} A_\si R^\eps_\si~=~\underbrace{\sum_{(w,\sigma)\text{ in Case a}} g_\si^\eps(w)}_{T_0(\eps)} ~+~\underbrace{
  \sum_{(w,\sigma)\text{ in Case b
}}  g_\si^\eps(w)}_{T_1(\eps)}.
\end{equation}

We will compute the limits of the two terms $T_0(\eps)$ and $T_1(\eps)$ in Lemmas~\ref{lem:T0} and~\ref{lem:T1}, respectively, and the proof of the Bethe ansatz will follow by summing the two results. Taking the limit in the expressions above is not straightforward: each term $g^\eps_\si(w)$ taken independently diverges like $O(1/\eps)$ as $\eps$ tends to 0 (since it contains a factor $L(z)$ or $M(z)$). For the analysis of both $T_0(\eps)$ and $T_1(\eps)$, we will use suitable groupings to cancel these diverging terms, and study the constant order terms that remain after these cancellations. In $T_0(\eps)$, we show that the diverging terms corresponding to the word $L\cdots L$ cancel with the diverging terms corresponding to the word $M\cdots M$, leaving an extra non-vanishing term that comes from the toroidal boundary conditions. In $T_1(\eps)$, the diverging part of $g_\si^\eps(w)$ cancels with the one of $g_{\si\circ(m,m+1)}^\eps(w)$ (where $m$ is such that $w_mw_{m+1}=ML$).

In the non degenerate case, we used several times the relation~\eqref{eq:6} in Lemma~\ref{lem:ratios} expressing  $\Theta(p_k,p_\ell)$ in terms of the functions $L,M$. This relation is particularly useful to compute ratios between different $A_\si$. When one entry is vanishing, we will use the following straightforward identity: for every $k\ge 2$ 
\begin{equation}
  \label{eq:19}
   \exp\big(i\Theta(0,p_{k})\big)=-\frac{L(z_k)}{M(z_k)}.
\end{equation}
 
Let us now move to the computation of the limits of the two terms in~\eqref{eq:10}.
We begin with the term $T_0(\eps)$ which is the easiest one  to compute.
 \begin{lemma}\label{lem:T0}
  We have 
   \begin{equation*}\label{eq:ConstantWordsSingular}
     \lim_{\eps\to 0} T_0(\eps)  ~=~ (2 - c^2) \Pi_M\, \psi(\vec x)~+~c^2N \Pi_M\sum_{\substack{\si\in\mathfrak S_n\\\si(n)=1}}A_\si Z_\si^{\vec x}.
   \end{equation*}
 \end{lemma}

\begin{proof}
For every permutation $\si\in\mathfrak S_n$, 
$$r^\eps_\si(M\cdots M)=M(z)\Pi_M.$$ Therefore, the contribution of the word $M\cdots M$ can be written as
  \begin{equation}\label{eq:17}
    \sum_{\si\in\mathfrak S_n}g_\si^\eps (M\cdots M)=M(z)\Pi_M\sum_{\si\in \mathfrak S_n}A_\si Z_{\si}^\eps(M\cdots M). 
  \end{equation}
  Let us now move to the contribution of the word $L\cdots L$. Using first~\eqref{eq:19}, and then Bethe's equations~\eqref{eq:BA} applied to $(0, p_2, \dots , p_n)$, we obtain
  \[
  \prod_{k=2}^n \frac{M(z_k)}{L(z_k)} = (-1)^{n-1} \exp\left(- i \sum_{k=2}^n\Theta(0,p_k)\right) \stackrel{\eqref{eq:BA}}{=} 1.
  \]
  Thus, for any permutation $\sigma$,
  \[
  r^{\varepsilon}_\si(L\cdots L) = L(z) \Pi_M.
  \]
  Defining $f(\si)=z^{x_{k-1}}$ when  $\si(k)=1$, we have 
  \begin{equation*}
    Z_\si^\eps(L\cdots L)=f(\si)Z_\si^{\tau^{-1}\vec x}.
  \end{equation*}
  Using the two displayed equations above and then the change of variables formula~\eqref{eq:5}, we can write the contribution of the constant word $L\cdots L$ as
  \begin{align}
    \sum_{\si\mathfrak S_n}g_\si^\eps(L\cdots L)&= L(z) \Pi_M\sum_{\si\in\mathfrak S_n}A_\si f(\si) Z_{\si}^{\tau^{-1}\vec x}\notag\\
    &=L(z)\Pi_M\sum_{\si\in\mathfrak S_n}A_\si f(\si\circ\tau^{-1}) Z_{\si}^{\vec x}\notag\\
    &=L(z)\Pi_M\sum_{\si\in\mathfrak S_n}A_\si z^{-N\mathbf 1_{\sigma(n)=1}} Z_{\si}^\eps(M\cdots M).\label{eq:22}
  \end{align}
Finally, putting the contributions~\eqref{eq:17} and~\eqref{eq:22}  of the two words together, we find
\begin{align*}
  T_0(\eps)&=\left[M(z)+L(z)\right]\Pi_M\sum_{\si\in\mathfrak S_n}A_\si  Z_{\si}^\eps(M\cdots M) +  \Pi_M\sum_{\si\in\mathfrak S_n}A_\si L(z)\left[1-z^{-N\mathbf 1_{\sigma(n)=1}}\right] Z_{\si}^\eps(M\cdots M).
\end{align*}
The proof follows by letting $z$ tend to 1 and using the straightforward computations 
\begin{align*}
L(z) + M(z) = 2 - c^2, \quad
\displaystyle\lim_{z\to1} Z_{\si}^\eps(M\cdots M)=Z^{\vec x}_\si \quad \text{ and }\quad
\displaystyle \lim_{z\to1} L(z)(1-z^{-N\mathbf 1_{\sigma(n)=1}})=c^2 N\mathbf 1_{\si(n)=1}.
\end{align*}
\end{proof}

The computation of the limit of $T_1(\eps)$ is less direct and requires further algebraic manipulations. 
Note that this limit corresponds to terms which cancel exactly in the non-degenerate case, but which contribute in this case. 

\begin{lemma}\label{lem:T1}
  We have that
  \begin{equation*}
    \label{eq:7}
    \lim_{\eps\to 0} T_1(\eps)~=~c^2 \Pi_M \Big(N+\sum_{k=2}^n\partial_1 \Theta(0,p_k)\Big)\, \psi(\vec x)~-~c^2N\Pi_M\,\sum_{\substack{\si\in\mathfrak S_n\\\si(n)=1}}A_\si Z_\si^{\vec x}.
  \end{equation*}
\end{lemma}
This lemma, together with Lemma~\ref{lem:T0}, implies the theorem in the singular case 
(note that the second term in the RHS above cancels exactly the second term in the expression of Lemma~\ref{lem:T0}).
The rest of the section is dedicated to proving Lemma~\ref{lem:T1}.


\begin{proof}
The proof is done in several steps.  
First, for $(w,\si)$ a word, permutation pair contributing to $T_1$, we will group $g_\si^\eps(w)$ with $g_{\si\circ(m,m+1)}^\eps(w)$ (for $m$ such that $w_mw_{m+1}=ML$) to cancel the singular terms in $g_\si^\eps(w)$ and $g_{\si\circ(m,m+1)}^\eps(w)$. While in the non-degenerate case, $g_\si^\eps(w)+g_{\si\circ(m,m+1)}^\eps(w)$ is exactly equal to 0, this is no longer the case here, and we will see that a new term written
$D_{\sigma}\hat g_\si(w)$ appears in the limit $\eps\searrow 0$, where
\begin{align}
	&\hat g_\si(w) 
	:= A_\si c^2\prod_{\substack{k : w_k = L,\\ \sigma(k)\ne 1}} L(z_{\sigma(k)}) 
	\prod_{\substack{k : w_k = M \\ \sigma(k) \neq 1}} M(z_{\sigma(k)})\ Z_\si(w) \quad \text{ and } \nonumber\\
	 &D_\sigma := \partial_1\Theta(0, p_{\sigma(m)}) + N \mathbf 1_{m = n}. \label{eq:AD}
\end{align}
Let us highlight the fact that we may see $m$ (and therefore $D_\sigma$) as a function of $\sigma$, since $m=\sigma^{-1}(1)-1$. 
In particular, $D_\sigma$ does not depend on $w$, as illustrated by the notation.

Second, we will show that $\hat g_\si(w)$ can be expressed in terms of $\hat g_{\si\circ [\ell,m]}(M\cdots M)$, where $[\ell,m]$ is a permutation depending on the word $w$ only. Finally, we will combine the two previous claims to conclude the proof.

  \bigbreak
  \noindent{\bf Claim 1. }{\em Fix $1\le m\le n$. Consider $\sigma\in\mathfrak S_n$ with $\sigma(m+1)=1$ and $w\in\mathcal W_1$ with $w_mw_{m+1}=ML$. Then,
\begin{equation*}
  \label{eq:3}
  \lim_{\eps\to 0} \, g_\si^\eps(w)+g_{\si\circ(m,m+1)}^\eps(w)~=~ D_\sigma\hat{g}_\sigma(w),
\end{equation*}
where $D_\sigma$ and $\hat{g}_\sigma(w)$ are defined in~\eqref{eq:AD}.
}
\bigbreak\noindent
{\em Proof of Claim 1.} To prove the claim, we write
\begin{equation}
  \label{eq:21} g_\si^\eps(w)+g_{\si\circ(m,m+1)}^\eps(w)=g_\si^\eps(w)\bigg(1+\frac{g_{\si\circ(m,m+1)}^\eps(w)}{g_\si^\eps(w)}\bigg)
\end{equation}
and establish the asymptotic behavior of the two terms in the product. First, note that  
 \begin{equation}
g^\varepsilon_\si(w) = \hat g_\si(w) \cdot \frac{M(z_{\sigma(m)})L(z) - 1}{c^2M(z_{\sigma(m)})}= \hat g_\si(w) \Big(\tfrac{1}{i\eps}+o\big(\tfrac1\eps\big)\Big).\label{eq:24}
\end{equation}
The computation of the ratio in~\eqref{eq:21} is similar to previous computations. It follows from the definitions of $r_\si^\eps(w)$ and $Z_\si^\eps(w)$ that 
\begin{equation*}
\frac{r^\eps_{\si\circ{(m,m+1)}}(w)}{r^\eps_\si(w)}=\frac{M(z)L(z_{\si(m)})-1}{M(z_{\si(m)})L(z)-1}\overset{~\eqref{eq:6}}=\exp\big(i\Theta(\eps,p_{\si(m)})\big),
\end{equation*}
and
\begin{equation*}
\frac{ Z^{\varepsilon}_{\sigma \circ (m, m+1)}(w)}{ Z^\eps_{\sigma}(w)}=\Big(\frac{z}{z_{\si(n)}}\Big)^{N\mathbf1_{m=n}}.
\end{equation*}
Furthermore, by Lemma~\ref{lem:ratios}, we have
\begin{equation*}
\frac{A_{\si\circ(m,m+1)}}{A_\si}=-\big({z_{\sigma(n)}}\big)^{N\mathbf 1_{m=n}}\exp\big(i\Theta(p_{\si(m)},0)\big)=-\big({z_{\sigma(n)}}\big)^{N\mathbf 1_{m=n}}\exp\big(-i\Theta(0,p_{\si(m)})\big). 
\end{equation*}
(We used that $\Theta(y,x)=-\Theta(x,y)$.) Using the three equations above and a Taylor expansion, we find
 \begin{equation*}
   \frac{g_{\si\circ(m,m+1)}^\eps(w)}{g_\si^\eps(w)}=-1+i\eps D_\si+o(\eps).
\end{equation*}
Plugging~\eqref{eq:24} and the equation above in~\eqref{eq:21} completes the proof of the claim. 
\hfill$\square$\bigbreak

\noindent{\bf Claim 2.} 
{\em For $w\in \mathcal W_1$, let $\ell$ and $m$ be the unique indexes such that $w_\ell w_{\ell+1}=LM$  and $w_mw_{m+1}=ML$. If $\sigma\in\mathfrak S_n$ satisfies $\sigma(m+1)=1$, we find that \begin{equation} 
 { \hat g_{\si}(w)}= \hat g_{\si\circ[\ell,m]}(M\cdots M),
\end{equation}
where $[\ell,m]$ is the permutation defined by
$$[\ell,m](i)=\begin{cases}\ell &\text{ if }i=m+1,\\
i-1&\text{ if }i\in\{m+2,m+3,\dots,\ell-1,\ell\},\\
i&\text{ otherwise}.\end{cases}$$
Here, as in the rest of the proof, we use periodic notation for the set $\{m+2,m+3,\dots,\ell-1,\ell\}$.}
\bigbreak
We will prove this step by ``zipping up" the letters $L$ in the word $w$ step by step: imagine a zipper positioned at $m+1$, the pre-image of $1$ by $\sigma$. This is also the position of the first letter $L$ after the series of $M$ in $w$. Move the zipper one step on the right, thus changing the first letter $L$ to $M$ in $w$ and composing $\sigma$ with the transposition exchanging the zipper index with the index on its right. Such a procedure will be shown to not affect the quantity $\hat g_\si(w)$. 
 By doing this again and again, we zip off all the letters $L$ and end up with the constant word $M\cdots M$. The composition of all the transpositions gives~$[\ell,m]$.
\bigbreak

\noindent{\em Proof of Claim 2.} 
Let us start with analysing one move of the zipper. We will show that, for any $1\le k\le n$ and any $w\in\mathcal W_1$ and $\sigma\in\mathfrak S_n$ satisfying $\sigma(k)=1$ and $w_{k-1}w_k=ML$, we have that
  \begin{equation}\label{eq:abd}
   { \hat g_{\si}(w)}= {\hat g_{\si\circ(k,k+1)}(w')},
  \end{equation}
  where $w'$ is the word obtained from $w$ by changing $w_k=L$ to $w'_k=M$.
  To prove this fact, first observe that,
 by definition, 
  \begin{equation}
    \label{eq:23}
    \frac{\hat g_{\si\circ(k,k+1)}(w')}{\hat g_{\si}(w)}=\frac{A_{\si\circ(k,k+1)}}{A_\si}\cdot\frac{M(z_{\si(k)})}{L(z_{\si(k)})}\cdot\frac{Z_{\si\circ(k,k+1)}(w')}{Z_\si(w)}.
  \end{equation}
By Lemma~\ref{lem:ratios} and~\eqref{eq:19}, we have 
\begin{equation}\label{eq:23a}
    \frac{A_{\si\circ(k,k+1)}}{A_\si}=-\left(z_{\si(1)}\right)^{-N\mathbf 1_{k=n}}\exp\big(i\Theta(0,p_{\si(k+1)})\big)=\left(z_{\si(1)}\right)^{-N\mathbf 1_{k=n}}\:\frac{L(z_{\si(k)})}{M(z_{\si(k)})}.
\end{equation}
Here, the ratio of the $Z$ functions is not {\em a priori} simple; however, our requirement that $\sigma(k) = 1$ implies that
\begin{equation}\label{eq:23b}
\frac{Z_{\sigma \circ (k, k+1)}(w')}{Z_{\si}(w)}  =  \left(z_{\sigma(1)}\right)^{N\mathbf 1_{k=n}}. 
\end{equation}
Plugging~\eqref{eq:23a} and~\eqref{eq:23b} in~\eqref{eq:23} implies~\eqref{eq:abd}.

To conclude observe that $[\ell,m]=(m+1,m+2)\circ(m+2,m+3)\circ\dots\circ(\ell-1,\ell)$ (note that the indexes are taken in $\bbZ_n$, so that $\ell$ may in fact be smaller than $m$). Applying~\eqref{eq:abd} repeatedly proves the claim. \hfill$\square$
\bigbreak

We can now conclude the proof of Lemma~\ref{lem:T1}. 
For each $w\in\mathcal W_1$, there are exactly two terms of the form $g_\si^\eps(w)$ entering in the sum $T_1(\eps)$. Claim 1 enables us to rewrite the limit of the sum of these two terms in terms of $D_\sigma$ and $\hat g_\si(w)$, so that 
$$\lim_{\eps\rightarrow 0}T_1(\eps)=\sum_{m}\sum_{\substack{\sigma\in\mathfrak S_n\\ \sigma(m+1)=1}}\ \sum_{\substack{w\in \mathcal W_1\\ 
w_{m-1}w_{m}=ML}} D_\sigma\hat g_\si(w).$$
For each word in the third sum, denote by $\ell$ the unique index such that $w_\ell w_{\ell+1}=LM$. Note that, as $w$ ranges over words in $\calW_1$ with $w_{m-1}w_{m}=ML$, $\ell$ takes all the values of $\{1,\dots,n\}$ different from $m$. Therefore, Claim 2 implies that 
$$
	\lim_{\eps\rightarrow 0}T_1(\eps)
	=\sum_{m}\sum_{\substack{\sigma\in\mathfrak S_n:\\ \sigma(m+1)=1}}\sum_{\ell:\,\ell\ne m} D_{\sigma}\hat g_{\si\circ[\ell,m]}(M\cdots M).
$$
Using the change of variables $\si\mapsto\si\circ[\ell,m]$ and exchanging the sums, we obtain 
$$\lim_{\eps\rightarrow 0}T_1(\eps)=\sum_{\ell}\sum_{\substack{\sigma\in\mathfrak S_n\\ \sigma(\ell)=1}}\hat g_{\si}(M\cdots M)\Big(\sum_{m:\,m\ne \ell} D_{\sigma\circ[\ell,m]^{-1}}\Big).$$
For any $\si$ as in the second sum, $m+1$ is sent to 1 by $\sigma\circ[\ell,m]^{-1}$. Thus
$$
\sum_{m:\,m\ne \ell} D_{\sigma\circ[\ell,m]^{-1}}
=\sum_{m:\,m\ne \ell}\big(\partial_1\Theta(0, p_{\sigma(m)}) + N \mathbf 1_{m = n}\big)
=N\mathbf 1_{\sigma(n)\ne 1}+\sum_{k=2}^n\partial_1\Theta(0,p_k),$$
where we used that $\sum_{m\ne \ell}\mathbf 1_{m=n}=\mathbf 1_{\ell\ne n}=\mathbf 1_{\sigma(n)\ne1}$.
Thus 
\begin{align*}
	\lim_{\eps\rightarrow 0}T_1(\eps)
	=\sum_{\sigma\in\mathfrak S_n}\hat g_{\si}(M\cdots M)
	\Big(N\mathbf 1_{\sigma(n)\ne 1}+\sum_{k=2}^n\partial_1\Theta(0,p_k)\Big).
\end{align*}
The proof follows by observing that $\hat g_\si(M\cdots M)=c^2\Pi_MA_\sigma Z_\si^{\vec x}$.
\end{proof}

\section{The six-vertex transfer matrix}\label{sec:V}

The goal of this section is to prove Proposition~\ref{prop:transfer-matrix}. We begin by defining the transfer matrix in a standard way. 
Let $G$ be the graph constructed by putting horizontal edges between neighboring vertices of $\mathbb Z_N$ together with vertical edges above and below each vertex. For basis vectors $\Psi_{\vec{x}}$ and $\Psi_{\vec{y}}$, let $\mathcal V(\Psi_{\vec{x}}, \Psi_{\vec{y}})$ be the set of arrow configurations on $G$ that obey the ice rule and whose vertical arrows coincide with $\Psi_{\vec x}$ on the bottom vertical edges and with $\Psi_{\vec y}$ on the top ones. Then, set
\[
\widetilde V(\Psi_{\vec{x}}, \Psi_{\vec{y}}) := \sum_{\omega \in \mathcal V(\Psi_{\vec{x}}, \Psi_{\vec{y}})} a^{n_1+n_2}b^{n_3+n_4}c^{n_5+n_6} \, ,
\]
where $n_i$ is the number of vertices of $\mathbb Z_N$ with configuration $i$ in $\omega$.  


\begin{lemma}\label{lem:Interlacement}
	For any pair of basis vectors $\Psi_{\vec{x}}, \Psi_{\vec{y}}$, 
	$$|\mathcal V(\Psi_{\vec{x}}, \Psi_{\vec{y}})|=\begin{cases} 2 &\text{ if $\vec x=\vec y$},\\
	1&\text{ if $\vec x\ne\vec y$ are interlaced},\\
	0&\text{ otherwise}.\end{cases}$$ 
\end{lemma}

\begin{proof}
If $\vec x=\vec y$, there are clearly two configurations of arrows in $\mathcal V(\Psi_{\vec x},\Psi_{\vec x})$ corresponding to all horizontal arrows pointing left, or all horizontal arrows pointing right.

Let us now assume that ${\vec{x}}\ne {\vec{y}}$ and $\mathcal V(\Psi_{\vec{x}}, \Psi_{\vec{y}})$ is nonempty. By summing the ice-rule at every vertex, we immediately see that $|\vec{x}| = |\vec{y}|$. 
Let $\om$ be a configuration in  $\mathcal V(\Psi_{\vec{x}}, \Psi_{\vec{y}})$.
Call a vertex in $\mathbb{Z}_N$ a source (resp. sink) if both vertical arrows enter (resp. exit) it; other vertices are neutral.

Since $\Psi_{\vec{x}} \neq \Psi_{\vec{y}}$, there exists $i$ such that $ \Psi_{\vec{x}}(i) = 1$ and $\Psi_{\vec{y}}(i) = -1$, or in other word such that  $i$ is a source. By the ice rule at $i$, the two horizontal arrows adjacent to $i$ must point outwards. 
This determines the orientation of three of the arrows adjacent to the vertex $i+1$. 
By considering the ice rule at the vertex $i+1$, we observe that the fourth adjacent arrow is also determined. 
Continuing this way, all arrows of $\om$ are determined, thus $|\mathcal V(\Psi_{\vec{x}}, \Psi_{\vec{y}})| = 1$. 

Moreover, the ice rule can be satisfied at $i+1$ if and only if $i+1$ is either neutral or a sink. 
If $i+1$ is neutral, $i+2$ must in turn be neutral or a sink; 
if $i+1$ is a sink, then $i+2$ must be neutral or a source. 
Repeating this reasoning at every vertex of $\mathbb{Z}_N$, we find that the configuration obeys the ice rule if and only if the sinks and sources alternate. Note that this alternation must be taken periodically -- i.e. if the first non-neutral vertex is a source, the final one must be a sink.
This translates immediately to the condition of interlacement between $\vec x$ and $\vec y$.
%

\end{proof}

\begin{corollary}\label{cor:MatrixEquality}
	The matrices $V$ defined in~\eqref{eq:V} and $\widetilde V$ are equal.\end{corollary}

\begin{proof}
We first consider the diagonal terms. 
As mentioned before, for any basis vector $\Psi_{\vec{x}}$, 
$V(\Psi_{\vec{x}}, \Psi_{\vec{x}})$ contains exactly two configurations: those with all horizontal arrows point in the same direction. 
In both such configurations, no vertex is of type 5 or 6, and their weight is 1. 
We conclude that all diagonal terms of $V$ are indeed equal to $2$.

For off-diagonal terms, the above lemma implies that $\widetilde V(\Psi_{\vec{x}}, \Psi_{\vec{y}})$ is zero if the two configurations are not interlacing, and equal to the weight of the single configuration in $\mathcal V(\Psi_{\vec{x}}, \Psi_{\vec{y}})$ if they are interlacing. 
Assuming $\Psi_{\vec{x}}$ and $\Psi_{\vec{y}}$ are interlacing, since $a = b =1$, the weight of the unique configuration $\om \in \mathcal V(\Psi_{\vec{x}}, \Psi_{\vec{y}})$ is obtained by counting the total number of vertices of type 5 and 6, i.e.~sources and sinks. 
Observe that $P(\Psi_{\vec{x}}, \Psi_{\vec{y}})$ is the number of sources and sinks. Thus, 
$\widetilde V(\Psi_{\vec{x}}, \Psi_{\vec{y}}) = c^{P(\Psi_{\vec{x}}, \Psi_{\vec{y}})} = V(\Psi_{\vec{x}}, \Psi_{\vec{y}})$.
\end{proof}

\begin{proof}[Proposition~\ref{prop:transfer-matrix}]
The symmetry and block-diagonal nature of $V$ is evident from the formula that defines its entries. 

Pick $M$ basis vectors $ \Psi_{\vec{x}_1}, \dots ,\Psi_{\vec{x}_M}$ in $\Omega$, and define $Z(c; \Psi_{\vec{x}_1}, \dots ,\Psi_{\vec{x}_M})$ to be the sum of the weights of all configurations on $\mathbb{T}_{N,M}$ whose vertical arrows configuration between the $i$th and $(i+1)$th row is equal to $\Psi_{\vec{x}_i}$ for all $i\in \bbZ_M$. By definition of $\widetilde V$ and the multiplicative nature of the weights,
\[
	Z(c; \Psi_{\vec{x}_1}, \dots ,\Psi_{\vec{x}_M}) 
	= \prod_{i=1}^M \widetilde V(\Psi_{\vec{x}_{i}}, \Psi_{\vec{x}_{i+1}})
	= \prod_{i=1}^M {V}(\Psi_{\vec{x}_{i}}, \Psi_{\vec{x}_{i+1}}) \, ,      
\]
where $\Psi_{\vec{x}_{M+1}} = \Psi_{\vec{x}_1}$. Summing over all possible configurations of the vertical arrows, we find that
\begin{align}
	Z_{6V}(1,1,c) & 
	= \displaystyle \sum_{\Psi_{\vec{x}_1}, \dots , \Psi_{\vec{x}_M}} Z(c; \Psi_{\vec{x}_1}, \dots, \Psi_{\vec{x}_M}) 
	= \sum_{\Psi_{\vec{x}_1}} V^M(\Psi_{\vec{x}_1}, \Psi_{\vec{x}_1}) 
	= \Tr(V^M) \, .  
\end{align}

\end{proof}

\section{The XXZ Model}\label{sec:XXZ}

Recall from Section~\ref{sec:notation} the notation related to the XXZ model, namely $w_i$, $H_i$ and the hamiltonian $H$. 
%
We prove Theorem~\ref{thm:XXZ} in two steps. We first show that there exists $E$ such that $H\psi = E \psi$ 
-- i.e. that either $\psi$ is an eigenvector of $H$, or that it is 0. Then, we compute the value of $E$. 
The first step follows from showing that the matrices $H$ and $V$ commute (where $V$ is the transfer matrix of the six-vertex model), 
and hence are simultaneously diagonalizable.

\begin{lemma}\label{lemma:CommutingMatrices}
	We have $
	VH = HV $.
\end{lemma}

\begin{proof}
	We have proved in Proposition~\ref{prop:transfer-matrix} that $V$ is a symmetric matrix. 
	Moreover, $H$ is also symmetric, as we observe from its definition. 
	Thus, it is sufficient to prove that $VH$ is itself a symmetric matrix to obtain the lemma. 

	For this proof, we suppress the dependence of basis vectors $\Psi_{\vec x}$ on $\vec{x}$ for notational convenience; 
	we will instead consider $\Psi$ as a function from $\mathbb{Z}_N$ to $\{-1,1\}$. 
	For such $\Psi$, define 
	\[
		I_\Psi := \left\{ i : \Psi(i) \neq \Psi(i+1) \right\}\, ,
	\] 
	where we use the periodic convention so that $N+1 \equiv 1$. 
	Then $w_i(\Psi)$ is nonzero whenever $i \in I_\Psi$. 
	Note that, due to periodicity, the cardinality of this set is always even. 
	
	Fix two basis vector $\Psi' \neq \Psi$. The definition of $H$ and the symmetry of $V$ imply that
	\begin{align}
		(VH)(\Psi, \Psi') - (VH)(\Psi', \Psi) 
		= & \Delta \cdot V(\Psi, \Psi') \left(|I_\Psi| - |I_{\Psi'}| \right) + \nonumber\\ 
		& + \displaystyle \sum_{i \in I_{\Psi'}} V(\Psi, w_i(\Psi')) -  \displaystyle \sum_{i \in I_{\Psi}} V(w_i(\Psi), \Psi')  \, .
		\label{eq:CommuteRelation}
	\end{align}
	Our aim is to prove that the above is zero for all $\Psi$ and $\Psi'$. 
	Recall the notation of Lemma~\ref{lem:Interlacement}: 
	we place the arrow configurations associated with $\Psi$ and $\Psi'$ below and above a copy of $\mathbb{Z}_N$, 
	denote this by $[\Psi, \Psi']$.
	A vertex $i$ is called a sink (resp. source) if $\Psi(i) \neq \Psi'(i)$ and $\Psi(i) = +1$ (resp. $-1$). Otherwise, $i$ is neutral. 
	Then $[\Psi, \Psi']$ is interlacing if and only if the sources and sinks alternate, considered periodically. 
	
	To start, assume that $\Psi$ and $\Psi'$ are not interlacing. 
	Then the first term in the right-hand side of~\eqref{eq:CommuteRelation} is null. 
	Moreover, all other terms also vanish 
	unless there exists some $i$ for which either $[w_i(\Psi),\Psi']$ or $[\Psi,w_i(\Psi')]$ is interlacing. 
	
	Suppose this is the case and consider one such $i$; by symmetry we can assume $[w_i(\Psi),\Psi']$ is interlacing. 
	Then the action of $w_i$ on $\Psi$ transforms an adjacent sink/source pair of  $[\Psi, \Psi']$ which occurred in the ``wrong'' order 
	into two neutral vertices, and thus creates the interlacing pair $[w_i(\Psi),\Psi']$. 
	Therefore, $\Psi(i) \neq \Psi'(i)$ and $\Psi(i+1) \neq \Psi'(i+1)$, and $i \in I_\Psi \cap I_{\Psi'}$. 
	Furthermore, if $[w_i(\Psi),\Psi']$ is interlacing, then so is $[\Psi,w_i(\Psi')]$, 
	as in both cases, the sink/source pair of $[\Psi, \Psi']$ is transformed by $w_i$ into a pair of neutral vertices. 
	Finally, both of these interlacing pairs have the same weight, as they exhibit the same number of sinks and sources.
	Thus their contribution to~\eqref{eq:CommuteRelation} cancels out, and by summing over $i$ the result is proved. 
	
	Assume now that $\Psi$ and $\Psi'$ are interlacing. 
	If $i \in I_\Psi \setminus I_{\Psi'}$, then $[\Psi, \Psi']$ has one neutral vertex and one sink or source at the vertices $i$ and $i+1$. 
	See Figure~\ref{fig:psipsi1} for the four possibilities.
	In this case, in $[w_i(\Psi), \Psi']$ we also find one of these four configurations at position $i,i+1$. 
	Moreover, the alternating sink/source structure is maintained in $[w_i(\Psi), \Psi']$, 
	and hence $w_i(\Psi)$ and $\Psi'$ are also interlacing. 
	Finally, the number of sources and sinks of $[w_i(\Psi), \Psi']$ is the same as in $[\Psi, \Psi']$, and we conclude that 
	$$ V(w_i(\Psi), \Psi') =  V(\Psi, \Psi').$$

	\begin{figure}
		\begin{center}
			\includegraphics[scale=0.9]{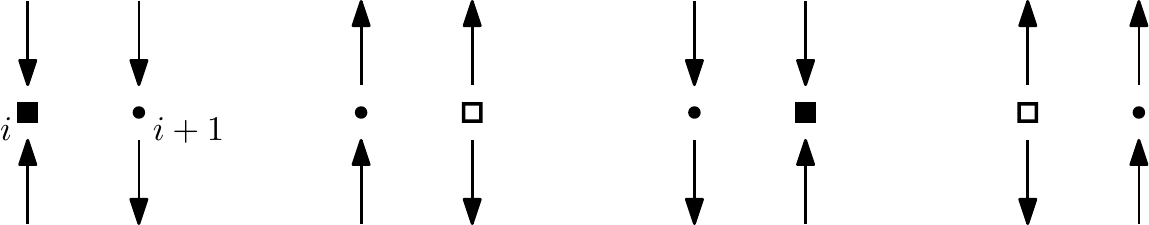}
			\caption{The four possibilities when $i \in I_{\Psi} \setminus I_{\Psi'}$. Here $\Psi$ is below and $\Psi'$ above. 
			The first and second configurations are mapped by $w_i$ to the third and fourth, respectively, and vice-versa.}
			\label{fig:psipsi1}
		\end{center}
	\end{figure}
	
	If $i \in I_\Psi \cap I_{\Psi'}$, there are two possible scenarios. 
	First, suppose $\Psi(i) \neq \Psi'(i)$. Then $[\Psi,\Psi']$ has a sink/source pair at $i$ and $i+1$. 
	Acting by $w_i$ on either configuration replaces the sink/source pair with a pair of neutral vertices.
	Thus, in this case 
	$$ V(w_i(\Psi), \Psi') = V(\Psi,w_i( \Psi')) = \frac{1}{c^2} V(\Psi, \Psi').$$
	Secondly, suppose that $\Psi(i) = \Psi'(i)$. 
	Then the vertices $i$, $i+1$ are both neutral in $[\Psi,\Psi']$, 
	and applying $w_i$ to either $\Psi$ or $\Psi'$ transforms them into a sink/source pair.
	The order in which this sink and source pair occurs in $[w_i(\Psi),\Psi']$ is reversed in $[\Psi,w_i(\Psi')]$. 
	Thus, exactly one of these pairs is interlacing, and we find that 
	\begin{align*}
		\text{either } &V(w_i(\Psi), \Psi') = {c^2} V(\Psi, \Psi') \text{ and } V(w_i(\Psi), \Psi') = 0\\
		\text{or } &V(\Psi, w_i(\Psi')) = {c^2} V(\Psi, \Psi') \text{ and } V(\Psi, w_i(\Psi')) = 0.
	\end{align*}
	
	Putting all this information together, through some straightforward algebra, we find that 
	\begin{align}\label{eq:csquaredsum}
		(VH)(\Psi, \Psi') - (VH)(\Psi', \Psi) 
		& =  V(\Psi, \Psi') \cdot \Big[ (\Delta -1) \left(|I_\Psi| - |I_{\Psi'}| \right)   + \nonumber \\ 
		&  + c^2 \cdot \sum_{i \in \overline{I}} 
		\left( \ind_{[\Psi, w_i(\Psi')] \text{ interlacing}} -  \ind_{[\Psi, w_i(\Psi')] \text{ interlacing}} \right)\Big]\, , 
	\end{align}
	where $\overline{I} = \{ i : i \in I_\Psi \cap I_{\Psi'}, \Psi(i) = \Psi'(i)\}$.
	Thanks to the definition of $\Delta$, the proof will be done if we can show that
	\[
		\displaystyle \sum_{i \in \overline{I}} 
		\left( \ind_{[\Psi, w_i(\Psi')] \text{ interlacing}} -  \ind_{[\Psi, w_i(\Psi')] \text{ interlacing}} \right) 
		= \frac{|I_\Psi| - |I_{\Psi'}|}{2} \, .
	\]
	Define a run $R \subset \mathbb{Z}_N$ to be a maximal connected subset of neutral vertices of $[\Psi,\Psi']$, taken periodically. 
	Let $l(R)$ and $r(R) +1$ be the vertices to the left (resp. right) of the run $R$.
	Note that, for $l(R)< i < r(R)$, 
	$i \in I_\Psi$ if and only if $i \in I_{\Psi'}$.
	Considering this, one may observe that the points of $\ovr I$ are exactly the points of $I_\Psi \cap I_{\Psi'}$ 
	that are contained inside some run. 
	The endpoints of the run however satisfy $l(R),r(R) \in I_\Psi \triangle I_{\Psi'}$.
	A short analysis shows that any point of $I_\Psi \triangle I_{\Psi'}$ is of the form $r(R)$ or $l(R)$ for some run $R$. 
	
	For a run $R$, write $N(R):=|R \cap \overline{I}|$. 
	Then the interlacement of $\Psi$ and $\Psi'$ implies that 
	\begin{align*}
		&l(R) \in I_\Psi \Longleftrightarrow r(R) \in \begin{cases} \ I_\Psi&\text{ if $N(R)$ is even},
 \\
\  I_{\Psi'}&\text{ if $N(R)$ is odd}.\end{cases}
	\end{align*}
	

	\begin{figure}[htb]
		\begin{center}
			\includegraphics{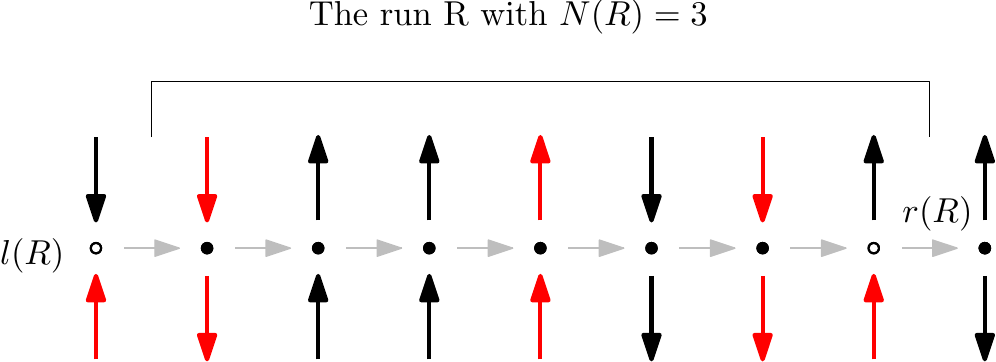}
		\end{center}
		\caption{An example of a run $R$. The red arrows are in $I_\Psi$ and $I_{\Psi'}$.}
	\end{figure}

	Fix a run $R$ and consider the sum 
	\begin{align}
		\label{eq:Rsum}
	 	\sum_{i \in \overline{I}\cap R} 
		\left( \ind_{[\Psi, w_i(\Psi')] \text{ interlacing}} -  \ind_{[\Psi, w_i(\Psi')] \text{ interlacing}} \right) \, .
	\end{align}
	Assume that $\Psi(l(R)) = +1$ and that $l(R) \in I_\Psi \setminus I_{\Psi'}$. 
	Let $i_1,i_2,\dots, i_{N(R)}$ be the elements of $\overline{I} \cap R$, ordered from left to right. 
	By assumption, there is a source at $l(R)$ and $\Psi(i_1) = \Psi'(i_1) =-1$. 
	It is easy to see that, in this case, the configuration $[\Psi, w_{i_1}(\Psi')]$ has a sink at $i_1$ and a source at $i_1 +1$, 
	and thus is an interlaced configuration. 
	However, $\Psi(i_2) = \Psi'(i_2) = +1$, and $[w_{i_2}(\Psi), \Psi']$ is interlaced - the opposite configuration as for $i_1$. 
	Repeating this, we see that the sum in~\eqref{eq:Rsum} is zero when $N(R)$ is even, and $+1$ when $N(R)$ is odd.
	In the former case, $r(R) \in I_{\Psi'}\setminus I_{\Psi}$, while in the latter, $r(R) \in I_{\Psi}\setminus I_{\Psi'}$.
	This is also valid when $N(R) =0$.
	
	The same procedure implies that, when $\Psi(l(R)) = +1$ and $l(R) \in I_{\Psi'}$, 
	the sum of~\eqref{eq:Rsum} is $-1$ when $N(R)$ is odd (thus when $r(R) \in I_{\Psi'}\setminus I_{\Psi}$) 
	and zero otherwise (i.e. when $r(R) \in I_{\Psi}\setminus I_{\Psi'}$). 
	The same analysis may be applied when $\Psi(l(R)) = -1$.
	In conclusion,
	\begin{align*}
		\sum_{R \text{ run }} \sum_{i \in \overline{I} \cap R } & \left( \ind_{[\Psi, w_i(\Psi')] \text{ interlacing}} -  \ind_{[\Psi, w_i(\Psi')] \text{ interlacing}} \right) \\ &= \displaystyle \sum_{R} \frac{(\ind_{l(R) \in I_{\Psi}} +  \ind_{r(R) \in I_{\Psi}}) - (\ind_{l(R) \in I_{\Psi'}} +  \ind_{r(R) \in I_{\Psi'}})}{2}\, .
	\end{align*}
	Since every element of $I_\Psi \triangle I_{\Psi'}$ is the boundary of some run, this completes the proof.
\end{proof}

\begin{proof}[Theorem~\ref{thm:XXZ}]
Lemma~\ref{lemma:CommutingMatrices} implies that $H \psi = E \psi$ for some $E$. It is sufficient to consider any individual coordinate to evaluate $E$. We choose to evaluate the coefficient of $\Psi_{(2,\dots,2n)}$, where we use the assumption $n \leq N/2$ to ensure that this coordinate vector is in $\Omega$. 
Thanks to the very simple structure of $H$,
we can explicitly compute the entry of $H\psi$ corresponding to $(2,\dots,2n)$:
\begin{align*}
	(H \psi)(2, \dots, 2n) 
	& = \frac{(N - 4 n) \Delta}{2} \cdot \psi(2, \dots, 2n)  + \\ 
	& + \sum_{i=1}^n \left[\psi(x_1, \dots, x_i - 1, \dots, x_n) + \psi(x_1, \dots, x_i +1, \dots, x_n)\right]\, .\nonumber    
\end{align*}
Now, thanks to the form of $\psi$, we deduce that
\[
(H \psi)(2, \dots, 2n) = \displaystyle \sum_{\sigma \in \mathfrak{S}_n} A_\sigma \prod_{k=1}^n z_{\sigma(k)}^{2k} \cdot \left[\frac{N \Delta}{2} - \sum_{k=1}^n \left(2 \Delta - \frac{1}{z_{\sigma(k)}} -  z_{\sigma(k)} \right) \right] \, . 
\]
The bracketed term is independent of $\sigma$, and is equal to 
\[
E = \frac{N \Delta}{2} - 2 \sum_{k=1}^n [\Delta - \cos(p_k)] \, ,
\]
as required.
\end{proof}

Note that the above computation is simple because of the choice of the coordinate $(2, \dots, 2n)$ and the simple action of the Hamiltonian. One may be tempted to reverse the procedure of this paper: prove that $H \psi = E \psi$ directly, and then look for a shrewd choice of coordinate to compute $\Lambda$. Unfortunately, the transfer matrix $V^{(n)}$ is far less well-behaved than $H$. Even in the most symmetric case, $(V^{(n)} \psi) (2, \dots, 2n)$ is a sum over exponentially many different coordinates of $\psi$ (as opposed to the linear number of terms above), many of which have dramatically different weights. 

\paragraph{Acknowledgements} The first and the third authors were funded by the IDEX grant of Paris-Saclay. The fifth author was funded by a grant from the Swiss NSF. All the authors are partially funded by the NCCR SwissMap.

\bibliographystyle{siam}
\bibliography{biblicomplete}

\end{document}